\documentclass[12pt]{amsart}
\usepackage[dvips]{color}
\usepackage{amsmath}
\usepackage{amsxtra}
\usepackage{amscd}
\usepackage{amsthm}
\usepackage{amsfonts}
\usepackage{amssymb}
\usepackage{eucal}
\usepackage{epsfig}
\usepackage{graphics}

\numberwithin{equation}{section}
\newtheorem{thm}{Theorem}[section]
\newtheorem{prop}[thm]{Proposition}
\newtheorem{lem}[thm]{Lemma}

\newtheorem{cor}[thm]{Corollary}

\newcommand{\nn}{\nonumber}

\newcommand{\bra}[1]{\langle #1 |}        
\newcommand{\ket}[1]{{| #1 \rangle}}      

\newcommand{\one}{\mathbf{1}}
\newcommand{\C}{{\mathbb C}}
\newcommand{\Z}{{\mathbb Z}}

\newcommand{\E}{{\mathcal E}}
\newcommand{\F}{\mathcal F}

\newcommand{\cP}{\mathcal{P}}
\newcommand{\cR}{\mathcal{R}}

\newcommand{\J}{J}
\newcommand{\bs}{\boldsymbol}

\newcommand{\GS}{\mathfrak{S}}
\newcommand{\gl}{\mathfrak{gl}}

\newcommand{\e}{\epsilon}

\newcommand{\la}{\lambda}
\newcommand{\La}{\Lambda}

\newcommand{\id}{{\rm id}}

\newcommand{\Ker}{\mathop{\rm Ker}}
\newcommand{\res}{{\rm res}}
\newcommand{\Res}{\mathop{\rm res}}

\newcommand{\Sym}{\mathrm{Sym}}
\newcommand{\Tr}{{\rm Tr}}

\newcommand{\gge}{\geqslant}
\newcommand{\lle}{\leqslant}

\newcommand{\Cp}{C^{\perp}}

\newcommand{\ep}{e^{\perp}}
\newcommand{\fp}{f^{\perp}}
\newcommand{\hp}{h^{\perp}}
\newcommand{\on}{\operatorname}
\newcommand{\mc}{\mathcal}

\begin{document}
\begin{title}[Quantum toroidal  $\mathfrak{gl}_1$ and Bethe ansatz]
{Quantum toroidal $\mathfrak{gl}_1$ and Bethe ansatz}
\end{title}
\author{B. Feigin, M. Jimbo, T. Miwa, and E. Mukhin}
\address{BF: National Research University Higher School of Economics, Russian Federation, International Laboratory of Representation Theory 
and \newline Mathematical Physics, Russia, Moscow,  101000,  Myasnitskaya ul., 20 and Landau Institute for Theoretical Physics,
Russia, Chernogolovka, 142432, pr.Akademika Semenova, 1a
}
\email{bfeigin@gmail.com}
\address{MJ: Department of Mathematics,
Rikkyo University, Toshima-ku, Tokyo 171-8501, Japan}
\email{jimbomm@rikkyo.ac.jp}
\address{TM: Institute for Liberal Arts and Sciences,
Kyoto University, Kyoto 606-8316,
Japan}\email{tmiwa@kje.biglobe.ne.jp}
\address{EM: Department of Mathematics,
Indiana University-Purdue University-Indianapolis,
402 N.Blackford St., LD 270,
Indianapolis, IN 46202, USA}\email{mukhin@math.iupui.edu}

\dedicatory{Dedicated to  Rodney Baxter on the occasion of his 75th birthday}

\begin{abstract} 
We establish the method of Bethe ansatz for the XXZ type model obtained from the 
$R$ matrix associated to quantum toroidal $\gl_1$. 
We do that by using shuffle realizations of the modules 
and by showing that the Hamiltonian of the model is obtained from a simple multiplication operator 
by taking an appropriate quotient. We expect this approach to be applicable to a wide variety of 
models.
\end{abstract}

\maketitle 

\section{Introduction}
The XXZ type models constitute 
a well-known large family of integrable quantum models, 
which was one of the main motivations for the very discovery of quantum groups. 

These models arise in the following algebraic setting.
We start with a quantum algebra with a triangular decomposition $\E=\E_>\otimes \E_0\otimes \E_<$,  
and an associated $R$ matrix $\cR$ in a completion of 
$\E_\gge \otimes\E_\lle$, where $\E_\gge=\E_>\otimes \E_0$ and 
$\E_\lle=\E_0\otimes \E_<$. We also fix a group like element $t$ in (a completion of) $\E$. 
For an $\E$ module $U$, we have the transfer matrix $T_U(t)=(1\otimes\on{Tr}_U)((1\otimes t)\cR)$, 
provided the trace is well-defined. 
The assignment $U\mapsto T_U(t)$ gives us a map from the Grothendieck ring of a suitable 
category of $\E$ modules to a completion of $\E_{\gge}$.
The standard properties of the trace and the $R$ matrix imply that 
this map is a ring homomorphism, and that the image is a commutative subalgebra. 
Given a suitable $\E$ module $V$ the commutative subalgebra of transfer 
matrices acts in $V$ and produces the XXZ type Hamiltonians associated to $\E$ and $V$. 

The problem of diagonalizing the action of the XXZ type Hamiltonians 
has been extensively studied for more than 80 years. It is done almost exclusively 
by the Bethe ansatz method. The approach is always the same: 
one writes a candidate for the eigenvector depending on auxiliary parameters 
in some explicit form--- the so called off-shell Bethe vector. 
Then one proves that if the parameters satisfy 
a system of algebraic equations, then the off-shell Bethe vector is indeed an eigenvector 
with an explicit eigenvalue. The system of equations is called Bethe equations and the corresponding 
eigenvector  is  called  Bethe vector. 
Then, in good situations, one proves that the Bethe ansatz is complete, 
meaning that the  Bethe vectors form a basis of the representation $V$ modulo explicit symmetries if any.

In this paper, we study the case of $\E$ being the quantum toroidal algebra of type $\gl_1$, 
also known as 
elliptic Hall algebra, $(q,\gamma)$ analog of $\mathcal{W}_{1+\infty}$, 
Ding-Iohara algebra, etc.. 
This algebra enjoys a wave of popularity 
due to its appearance in geometry \cite{BS}, \cite{FT}, \cite{S}, 
\cite{SV1},\cite{SV2} and in integrable systems \cite{FKSW}, \cite{FKSW2}, \cite{KS}.

It appears that the known methods of finding off-shell Bethe vectors
are not directly applicable to $\E$.
We propose an alternative way to obtain the spectrum of the Hamiltonians. 
The idea is to introduce an appropriate space of functions, 
and to identify the Hamiltonians with the projection of simple operators of multiplication by 
symmetric functions. 
The Bethe equations appear naturally as the condition for describing the kernel of the projection.
We expect that this method can be applied to many cases 
including the ones where the standard Bethe ansatz technique is already established. 

\medskip

Let us describe the logic of our approach in more detail. The quantum toroidal $\gl_1$ algebra $\E$ depends on complex parameters $q_1,q_2,q_3$ such that $q_1q_2q_3=1$. The algebras $\E_>$, $\E_<$ and $\E_0$ are generated by currents $e(z)$, $f(z)$ and $\psi^\pm(z)$ (plus central elements and their duals) respectively. The commutation relations are similar to that of the 
quantum affine $\mathfrak{sl}_2$ algebra, but they are written 
in terms of the cubic polynomial $g(z,w)=(z-q_1w)(z-q_2w)(z-q_3w)$, 
see Section \ref{toroidal-algebra}. 
There is a projective action of the  
group $\on{SL}(2,\Z)$ on $\E$ by automorphisms. 
Along with the initial currents $e(z)$, $f(z)$ and $\psi^\pm(z)$, 
we also use the currents $e^\perp(z)$, $f^\perp(z)$ and $\psi^{\pm,\perp}(z)$ obtained by 
applying the automorphism from $\on{SL}(2,\Z)$ corresponding to the rotation by 90 degrees
\cite{BS},\cite{M}. We call them ``perpendicular" currents.

We define the coproduct and the $R$ matrix in terms of perpendicular currents. 
We consider modules which are lowest weight
modules with respect to initial currents, namely 
modules generated by a vector  $\ket{\emptyset}$ such that
\begin{align}
f(z)\ket{0}=0,\quad \psi^\pm(z)\ket{0}=\phi(z)\ket{0}, 
\label{lwt-cond}
\end{align}
where $\phi(z)$ is a rational function. 
The most important example is a family of Fock modules $\F(u)$ 
depending on a complex parameter $u$. 
These modules are irreducible under the Heisenberg subalgebra of $\E$ 
generated by perpendicular currents $\psi^{\pm,\perp}(z)$.
Other perpendicular currents $e^\perp(z)$, $f^\perp(z)$
act in $\F(u)$ by vertex operators \cite{FKSW}, 
while operators $\psi^{\pm}(z)$ can be identified with Macdonald operators \cite{FHHSY}, \cite{FFJMM2}.

Among the Hamiltonians of the model, the simplest is the degree one term  $H_p$
of the transfer matrix $T_{\F(u)}(p^{d^\perp})$, where $p\in\C$ and $d^\perp$ is the degree operator 
counting $e(z)$ as $1$ and $f(z)$ as $-1$, see Lemma \ref{int expl}.
It turns out that $H_p$ coincides with the operator considered in 
\cite{FKSW},\cite{FKSW2},\cite{KS} in relation to the deformed Virasoro algebra. 
Operator $H_p$ acting in a generic tensor product of Fock modules 
for generic $p$ has simple spectrum, and we do not consider other Hamiltonians in the present paper.

Following the ideas of \cite{FO}, \cite{Ng}, we realize algebra $\E_>$ in an appropriate space of functions $Sh_{0}$, see Section \ref{shuffle sec}.
 We also introduce another space of functions $Sh_{1}(u)$ together with left and right 
actions of algebra $Sh_{0}$. 
Moreover, we extend the left action to the action of $\E$. 
We denote $\J_0$  the image of the right action: $\J_0=Sh_{1}(u)Sh_{0}'$,  
where the prime denotes the augmentation ideal. We use certain filtration, 
see Appendix \ref{gordon sec}, to prove that the quotient $Sh_{1}(u)/\J_0$ 
is isomorphic to the Fock module $\F(u)$ as $\E$ module. We introduce a subspace $N$ 
of functions in $Sh_{1}(u)$ defined by certain regularity conditions and show that  
$N\oplus \J_0=Sh_{1}(u)$, see Section \ref{L sec}. 
Under a natural embedding of $Sh_{1}(u)$ to a completion of $\E_{\gge}$, the space $N$ is identified with the space of matrix elements of $L$ operators of the form $L_{\emptyset,v}=(1\otimes \bra{\emptyset}) \cR (1\otimes v)$, $v\in\F(u)$. Moreover under the projection $Sh_{1}(u)\to Sh_{1}(u)/\J_0=\F(u)$, the function corresponding to $L_{\emptyset,v}$ is mapped to $v$.

The coefficients of the series $\psi^+(z)$ act in the space of functions 
$Sh_{1}(u)$ by multiplications by symmetric polynomials. 
It is easy to see that $H_0=\lim_{p\to 0}H_p$ coincides (up to an explicit constant) 
with the linear term $h_1$ of $\psi^+(z)$, and in particular that, 
in the subspace $Sh_{1,n}(u)$ of functions in $n$ variables, $H_0$ acts 
simply by multiplication by $\sum_{i=1}^nx_i$ (up to multiplicative and additive explicit constants),  
see Theorem \ref{main}, \eqref{h1 act}, \eqref{first-int}. 
In the limit $p\to 0$ the algebra of all Hamiltonians of 
the model coincides with the algebra generated by coefficients of $\psi^+(z)$. 

Finally, we define the space of $p$-commutators: 
$\J_p=\{Sh_{1}(u)g -p^{\deg g}g\,Sh_{1}(u) \mid\ g\in Sh_{0}'\}$. 
The multiplication by symmetric polynomials  clearly preserves 
this space,  
and for generic $p$ we have the direct sum decomposition of vector spaces: 
$N\oplus \J_p=Sh_{1}(u)$. 

\medskip

Our principal result is: {\it the projection of operator $H_0$ acting in $Sh_{1}(u)$ to the space of matrix elements of $L$ operators $N$ along space $\J_p$ coincides with Hamiltonian $H_p$ acting on $\F(u)=N$.}  In other words,  $\on{Pr}_{\J_0} H_p v=\on{Pr}_{\J_p} H_0 v$ for all $v\in N$, see Theorem \ref{main}. 

\medskip

This identification immediately leads to the Bethe equation and the computation of the spectrum. 

Namely, we consider the dual space to $Sh_{1}(u)$ 
and evaluation functionals defined as evaluation of functions in $Sh_{1}(u)$ 
at fixed complex numbers $\{a_i\}$. 
Such a functional is obviously an eigenvector with respect  
to multiplication by a function $f$, with the eigenvalue given by evaluation of $f$ 
at $\{a_i\}$. Also clearly, the evaluation functional has $\J_p$ in the kernel if and 
only if the evaluation numbers $\{a_i\}$ satisfy the Bethe equation 
\begin{align*}
\phi(a_i)\prod_{j(\neq i)} \frac{g(a_j,a_i)}{g(a_i,a_j)}=p^{-1}
\end{align*}
for all $i$, see (\ref{BAE}) and Theorem \ref{spectrum}, 
where $\phi(z)$ is the weight of the module
in \eqref{lwt-cond}.
Therefore, we obtain a description of the spectrum of the Hamiltonians 
in the dual module $V=\F(u)^*$.

We also study the off-shell Bethe vector. 
A result of \cite{FHSSY} allows us to write
the canonical element of $\F(u)^*\otimes \F(u)$ in the form
$\sum_{\la}\bra{\la}\otimes f_\la(x)\in \F(u)^*\otimes
N$
with explicit functions $f_\la(x)$. 
The latter is the off-shell Bethe vector, from which the 
Bethe vector is obtained by evaluating the second component at $\{a_i\}$. 
We give the result in Proposition \ref{off-shell}. 

In this paper we consider only tensor products of Fock spaces of $\E$, but we expect such a scheme can be used for many modules over many quantum algebras. Also we skip the question of the completeness of the Bethe ansatz here, but we expect it can be proved for generic $p$ by deforming the $p=0$ evaluation 
maps $\rho_\la$ 
described in Appendix \ref{gordon sec} in a standard way.

\medskip

We note that operators $\psi^{\pm}(z)$ acting in $\F(u)$ can be identified with operators acting in equivariant $K$-theory of the Hilbert scheme of points on $\C^2$, where $q_1,q_3$ are equivariant parameters. Then the algebra of the XXZ type 
Hamiltonians \{$T_V(p^{d^\perp})$
\} provides the deformation of these operators which is expected to be related to ``quantum equivariant $K$-theory". 
Such an interpretation was one of motivations for our work.

In the conformal limit, algebra $\E$
becomes the $W$ algebra, and the corresponding integrals of motion in relation 
with Bethe equations were studied in \cite{L}, \cite{AL}. 
Another Hamiltonian of similar kind was considered in \cite{Sa1}, \cite{Sa2}. 
We also feel that there is some connection to the work
\cite{NS}, where the authors find a connection between 
Bethe ansatz and supersymmetric gauge theory.

\medskip

The paper is constructed as follows.
In Section \ref{toroidal-algebra} we describe algebraic properties of quantum toroidal $\gl_1$ algebra and the Fock module. In Section \ref{shuffle sec} we establish functional realizations of $\E_>$ and the Fock module. For that we use Gordon filtration established in Appendix \ref{gordon sec}. We study matrix elements of $L$ operators and their relation to the shuffle algebras in 
Section \ref{L oper sec}. 
In Section \ref{Bethe sec} we describe the XXZ type Hamiltonians, compute explicitly the first one and diagonalize it. In Section \ref{mult fock sec} we extend our results to the case of tensor product of Fock modules.

\section{Quantum toroidal $\mathfrak{gl}_1$}\label{toroidal-algebra}

In this section, we introduce our notation concerning the quantum toroidal $\mathfrak{gl}_1$ algebra. 

\subsection{Algebra $\E$}
Fix complex numbers $q, q_1,q_2,q_3$ satisfying $q_2=q^2$ and $q_1q_2q_3=1$. 
We assume further that, for integers $l,m,n\in\Z$,  
$q_1^lq_2^mq_3^n=1$ holds only if $l=m=n$. 
We set 
\begin{align*}
& g(z,w)=(z-q_1w)(z-q_2w)(z-q_3w),\\
&\kappa_r=(1-q_1^r)(1-q_2^r)(1-q_3^r).
\end{align*}

The quantum toroidal algebra of type $\gl_1$,  which we denote by  $\E$,  
is a $\C$-algebra 
generated by elements
\begin{align*}
e_k,\ f_k \quad (k\in \Z),\quad h_r\quad (r\in \Z\backslash\{0\})
\end{align*}
and invertible elements
\begin{align*}
C,\ \Cp,\ D,\ D^\perp,
\end{align*}
subject to the relations given below.
We write them in terms of the generating series
\begin{align*}
&e(z) =\sum_{k\in \Z} e_{k}z^{-k}, \quad 
f(z) =\sum_{k\in\Z} f_{k}z^{-k}, \\
&\psi^{\pm}(z) = (\Cp)^{\mp 1} 
\exp\bigl(\sum_{r=1}^\infty \kappa_r h_{\pm r}z^{\mp r}\bigr)\,.
\end{align*}
The defining relations of $\mathcal{E}$ read as follows. 
\begin{gather*}
\text{$C$, $\Cp$ are central},\quad
D D^{\perp}=D^{\perp} D,\\
D e(z)=e(qz)D,\quad D f(z)=f(qz)D,\quad 
D \psi^\pm(z) =\psi^\pm(qz)D,\\
D^{\perp}e(z)=q e(z)D^{\perp},\quad
D^{\perp}f(z)=q^{-1} f(z)D^{\perp},\quad
D^\perp\psi^{\pm}(z)=\psi^{\pm}(z)D^{\perp},
\\
\psi^\pm(z)\psi^\pm(w) = \psi^\pm(w)\psi^\pm (z), 
\\
\frac{g(C^{-1}z,w)}{g(C z,w)}\psi^-(z)\psi^+ (w) 
=
\frac{g(w,C^{-1}z)}{g(w,C z)}\psi^+(w)\psi^-(z),
\\
g(z,w)\psi^\pm(C^{(-1\mp1)/2}z)e(w)
+g(w,z)e(w)\psi^\pm(C^{(-1\mp1) /2}z)=0,
\\
g(w,z)\psi^\pm(C^{(-1\pm1)/2}z)
f(w)+g(z,w)f(w)\psi^\pm(C^{(-1\pm1)/2}z)=0\,,
\\
[e(z),f(w)]=\frac{1}{\kappa_1}
(\delta\bigl(\frac{Cw}{z}\bigr)\psi^+(w)
-\delta\bigl(\frac{Cz}{w}\bigr)\psi^-(z)),\\
g(z,w)e(z)e(w)+g(w,z)e(w)e(z)=0, \\
g(w,z)f(z)f(w)+g(z,w)f(w)f(z)=0,\\
\mathop{\mathrm{Sym}}_{z_1,z_2,z_3}z_2z_3^{-1}
[e(z_1),[e(z_2),e(z_3)]]=0\,,\\
\mathop{\mathrm{Sym}}_{z_1,z_2,z_3}z_2z_3^{-1}
[f(z_1),[f(z_2),f(z_3)]]=0\,.
\end{gather*}

In particular we have the relations
\begin{align}
&[h_{r},e_n]=-\frac{1}{r}\, e_{n+r}\, C^{(-r-|r|)/2}\,,
\label{he}\\
&[h_{r},f_n]=\frac{1}{r}\, f_{n+r}\, C^{(-r+|r|)/2}\,,
\label{hf}\\
&[h_{r},h_{s}]=
\delta_{r+s,0}\,\frac{1}{r} \frac{C^{r}-C^{-r}}{\kappa_r}\,,
\label{hh}
\end{align}
for all $r,s\in\Z\backslash\{0\}$ and $n\in\Z$. 

The subalgebra of $\E$ generated by
$e_n,f_n$ ($n\in\Z$), $h_r$ ($r\in\Z\backslash\{0\}$) and $C,C^\perp$  
will be denoted by $\E'$.  

\begin{figure}[t]
\setlength{\unitlength}{1mm}
{
\begin{picture}(60,40)(-10,20)
\put(22,22){$e_0$}
\put(-5,35){$h_{-1}$}
\put(45,35){$h_{1}$}
\put(22,48){$f_0$}
\put(15,27){\vector(-2,1){10}}
\put(40,32){\vector(-2,-1){10}}
\put(5,40){\vector(2,1){10}}
\put(30,45){\vector(2,-1){10}}
\end{picture}
}
\caption{Automorphism $\theta$.}\label{fig1}
\end{figure}


Algebra $\E$ admits an automorphism $\theta$ of order 4 
\cite{BS,M} such that (see Fig.\ref{fig1}) 
\begin{align}
\theta&:e_0\mapsto h_{-1},\ 
h_{-1}\mapsto f_0,\ f_0\mapsto h_1,\ h_1\mapsto e_0,
\label{theta}\\
&\Cp\mapsto C\,,\ C\mapsto (\Cp)^{-1}\,,\
D^\perp \mapsto D\,,\ D\mapsto (D^\perp)^{-1}\,.
\nn
\end{align}
Quite generally, we write $x^\perp=\theta^{-1}(x)$ 
for an element $x\in \E$.  In this notation
\begin{align*}
\ep_0=h_1,\quad \fp_0=h_{-1},\quad \hp_1=f_0,\quad \hp_{-1}=e_0.
\end{align*}
The relations \eqref{he}--\eqref{hh} imply further that 
\begin{align*}
&\ep_1=f_1\Cp,\quad \ep_{-1}=e_1C^{-1},\\
&\fp_1=f_{-1}C,\quad \fp_{-1}=e_{-1}(\Cp)^{-1},
\\
&\ep_{m+1}=[\ep_m,f_0]\Cp,\quad
\ep_{-m-1}=[e_0,\ep_{-m}], \\
&\fp_{m+1}=[f_0,\fp_m],\quad \fp_{-m-1}=[\fp_{-m},e_0](\Cp)^{-1}.
\end{align*}

Algebra $\E$ is equipped with a $\Z^2$ grading defined by the assignment 
\begin{align}
&\deg e_n=(1,n),\quad
\deg f_n=(-1,n),\quad
\deg h_r=(0,r),
\label{grading1}\\
&
\deg x=(0,0)\quad (x=C,C^{\perp},D,D^\perp).
\label{grading2}
\end{align}
We have 
\begin{align*}
&\deg \ep_n=(-n, 1),
\quad \deg \fp_n=(-n,-1),\quad \deg \hp_r=(-r,0).
\end{align*}
For a homogeneous element $x\in\E$ with $\deg x=(n_1,n_2)$, 
we say that $x$ has {\it principal degree} $n_1$ and {\it homogeneous degree} 
$n_2$ and write 
\begin{align*}
\mathrm{pdeg}\, x=n_1,\quad \mathrm{hdeg}\, x=n_2\,. 
\end{align*} 
Note that 
$D^\perp x(D^\perp)^{-1}=q^{\mathrm{pdeg}\, x}x$ and 
$DxD^{-1}=q^{-\mathrm{hdeg}\, x}x$.
 
\begin{figure}[t]
\setlength{\unitlength}{.8mm}
\begin{picture}(120,120)(-40,-60)
\put(6,55){$\vdots$}\put(26,55){$\vdots$}\put(46,55){$\vdots$}
\put(17,58){\line(0,-1){65}}
\put(17,12){\line(1,0){20}}
\put(17,-7){\line(1,0){20}}
\put(37,12){\line(0,-1){67}}
\put(-56,54){\vector(0,-1){5}}
\put(-56,54){\vector(1,0){5}}
\put(-49,53){$n_2$}
\put(-57,45){$n_1$}
\put(-16,54){\oval(10,8)}
\put(-18,52){$\E_{\lle}^\perp$}
\put(66,54){\oval(10,8)}
\put(64,52){$\E_{\lle}$}
\put(66,-52){\oval(10,8)}
\put(64,-54){$\E_{\gge}^\perp$}
\put(-16,-52){\oval(10,8)}
\put(-18,-54){$\E_{\gge}$}
\put(5,40){$\fp_2$}\put(25,40){$\hp_2$}
\put(45,40){$\ep_2$}
\put(-40,20){$\cdots$}\put(-20,20){$f_{-2}$}
\put(-2,20){$f_{-1}$, }\put(6,20){\ $\fp_1$}
\put(20,20){$f_0$, }\put(25,20){\ $\hp_1$}
\put(40,20){$f_1$, }\put(45,20){\ $\ep_1$}
\put(60,20){$f_2$}\put(80,20){$\cdots$}
\put(-40,00){$\cdots$}\put(-20,00){$h_{-2}$}
\put(-2,00){$h_{-1}$}\put(6,00){$, \fp_{0}$}
\put(24,00){$\bullet$}
\put(40,00){$h_1$}\put(45,00){$, \ep_0$}
\put(60,00){$h_2$}\put(80,00){$\cdots$}
\put(-40,-10){\line(1,0){75}}
\put(35,-10){\line(0,1){19}}
\put(15,-10){\line(0,1){19}}
\put(15,9){\line(1,0){65}}
\put(-40,-20){$\cdots$}\put(-20,-20){$e_{-2}$}
\put(-2,-20){$e_{-1}$}\put(6,-20){$, \fp_{-1}$}
\put(20,-20){$e_0$}\put(25,-20){$, \hp_{-1}$}
\put(40,-20){$e_1$}\put(45,-20){$, \ep_{-1}$}
\put(60,-20){$e_2$}\put(80,-20){$\cdots$}
\put(5,-40){$\fp_{-2}$}
\put(25,-40){$\hp_{-2}$}
\put(45,-40){$\ep_{-2}$}
\put(7,-55){$\vdots$}\put(27,-55){$\vdots$}\put(47,-55){$\vdots$}
\end{picture}
\caption{Subalgebras $\E_{\gge}$, $\E_{\lle}$, $\E_{\gge}^\perp$, $\E_{\lle}^\perp$. 
The elements $C, C^\perp,D,D^\perp$ 
placed at the center $\bullet$ are common to all these subalgebras.}
\label{fig2}
\end{figure}

Introduce the following subalgebras:
\begin{align*}
&\E_{\gge}=\langle e_n\ (n\in \Z), \ h_{r}\ (r>0),\ C,\Cp, D, D^\perp
 \rangle \,,
\\
&\E_{\lle}=\langle f_n\ (n\in \Z), \ h_{-r}\ (r>0),\ C,\Cp, D, D^\perp
\rangle \,,
\\
&\E_{\gge}^\perp=\langle \ep_n\ (n\in \Z), \ \hp_{r}\ (r>0),\ C,\Cp, D, D^\perp  
\rangle \,,
\\
&\E_{\lle}^\perp=\langle \fp_n\ (n\in \Z), \ \hp_{-r}\ (r>0),\ \ C,\Cp, D, D^\perp  
\rangle\,.
\end{align*}
We picture generators of algebra $\E$ and their perpendicular counterparts on a plane according to their grading. 
The subalgebras $\E_{\gge},\E_{\lle},\E_{\gge}^\perp, \E_{\lle}^\perp$ 
are generated by the elements appearing 
respectively in the lower, upper, right 
and left half plane, see Fig. \ref{fig2}. 

We set also
\begin{align*}
&\E_{>}=\langle e_n\ (n\in \Z) \rangle\,,
\quad \E_{<}=\langle f_n\ (n\in \Z) \rangle\,,
\\
&\E_{>}^\perp=\langle \ep_n\ (n\in \Z) \rangle\,,
\quad \E_{<}^\perp=\langle \fp_n\ (n\in \Z) \rangle\,.
\end{align*}
One can easily check that $h^\perp_{-r},Ce^\perp_{-r}\in\E_>$
for $r>0$.

\subsection{Bialgebra sturcture and $R$ matrix}        
Algebra $\E$ is endowed with a topological bialgebra structure.
We choose the following coproduct $\Delta$ and counit $\varepsilon$, 
defined in terms of the perpendicular generators,
\begin{align}
&\Delta(e^\perp_n)=\sum_{j\ge0}e^\perp_{n-j}\otimes 
\psi^{+,\perp}_j\bigl(C^\perp\bigr)^n+1\otimes e^\perp_n\,,
\label{copro1}\\
&\Delta(f^\perp_n)=f^\perp_n\otimes 1+
\sum_{j\ge0}\psi^{-,\perp}_{-j}\bigl(C^\perp\bigr)^n\otimes f^\perp_{n+j}\,,
\label{copro2}\\
&\Delta h_{r}^\perp=h^\perp_{r}\otimes 1+\bigl(C^\perp\bigr)^{-r}\otimes h^\perp_{r}
\,,
\label{copro3}\\
&\Delta h_{-r}^\perp=h^\perp_{-r}\otimes \bigl(C^\perp\bigr)^r+1\otimes h^\perp_{-r}
\,,
\label{copro4}\\
&\Delta\, x=x\otimes x\quad 
(x=C,\Cp,D,D^\perp)\,,
\label{copro5}
\\
&\varepsilon(e^\perp_n)=\varepsilon(f^\perp_n)=0,
\quad \varepsilon(h^\perp_{\pm r})=0\,,
\quad 
\varepsilon(x)=1\quad (x=C,\Cp,D,D^\perp)\,,
\label{counit}
\end{align}
for all $n\in \Z$ and $r>0$. 
Here we set 
$\psi^{\pm,\perp}(z)=\sum_{\pm j\ge0}\psi^{\pm,\perp}_jz^{-j}$,
$\psi^{\pm,\perp}_0=C^{\pm 1}$.

\medskip 

Quite generally, a bialgebra pairing on a bialgebra $A$
is a symmetric non-degenerate bilinear form $(~,~):A\times A\to \C$ 
with the properties
\begin{align*}
&(a,b_1b_2)=(\Delta(a), b_1\otimes b_2),\quad
(a,1)=\varepsilon(a)\,
\end{align*}
for any $a,b_1,b_2\in A$. 
With each such pair $(A,(~,~))$, there is an associated bialgebra
$DA$ called the Drinfeld double of $A$. 
As a vector space $DA=A\otimes A^{\mathrm{op}}$,  
where $A^{\mathrm{op}}$ is a copy of $A$ endowed with the opposite coalgebra
structure. Moreover 
$A^+=A\otimes 1$ and $A^-=1\otimes A^{\mathrm{op}}$ are sub bialgebras of $DA$, 
and the commutation relation
\begin{align*}
\sum (a_{(2)},b_{(1)})\, a^-_{(1)}b^+_{(2)}
=\sum (b_{(2)},a_{(1)})\, b^+_{(1)}a^-_{(2)}
\end{align*} 
is imposed for $a,b\in A$. Here $a^+=a\otimes 1$, $a^-=1\otimes a$, and 
we use the Sweedler notation 
$\Delta(a)=\sum a_{(1)}\otimes a_{(2)}$ for the coproduct.  
The canonical element of $DA=A\otimes A^{\rm op}$ considered as 
an element of a suitable completion of $A^+\otimes A^-\subset DA\otimes DA$
is called the universal $R$ matrix and is denoted by $\cR$.
It has the properties
\begin{align}
&\cR\ \Delta(x)=\Delta^{\mathrm{op}}(x)\ \cR\quad (x\in DA)\,,
\label{int-R}\\
&\bigl(\Delta\otimes\id\bigr)\cR=\cR_{1,3}\,\cR_{2,3}\,,
\quad
\bigl(\id\otimes\Delta\bigr)\cR=\cR_{1,3}\,\cR_{1,2}\,,
\label{copro-R}\\
&\cR_{1,2}\cR_{1,3}\cR_{2,3}=\cR_{2,3}\cR_{1,3}\cR_{1,2}\,,
\label{YBE}
\end{align}
where, as usual, the suffixes $i,j$ of $\cR_{i,j}$ 
stand for the tensor components, e.g., $\cR_{1,2}=\cR\otimes 1$. 

The bialgebra $A=\E_{\gge}^\perp$ has a bialgebra pairing such that the non-trivial
pairings of the generators are given by
\begin{align*}
&(\ep_m,\ep_n)=\frac{1}{\kappa_1}\delta_{m,n}\,,\quad
(\hp_r,\hp_{s})=\frac{1}{r\kappa_r}\delta_{r,s}\,,\\
&(C,D)=(C^\perp,D^\perp)=q^{-1}\,.
\end{align*}
This pairing respects the $\Z^2$ grading
in the sense that $(a,b)=0$ unless $\deg a=\deg b$.
We identify $A^{\mathrm{op}}$ with $\E_{\lle}^\perp$ through the following
isomorphism of algebras, which is also an anti-isomorphism of coalgebras, 
\begin{align*}
e^\perp_n\mapsto f^\perp_{-n},\quad 
h^\perp_r\mapsto h^\perp_{-r},
\quad x\mapsto x^{-1}\quad (x=C,C^\perp,D,D^\perp).
\end{align*}
The Drinfeld double of $\E_{\gge}^\perp$ is then identified with
 $\E_{\gge}^\perp\otimes \E_{\lle}^\perp$. 
Its quotient by the relation $x\otimes 1=1\otimes x$ 
($x=C,C^\perp,D,D^\perp$) is isomorphic to the algebra $\E$ \cite{BS}.  

The universal $R$ matrix is an element of a certain completion of 
$\E_{\gge}^\perp\otimes \E_{\lle}^\perp\subset \E\otimes \E$, with the structure 
\begin{align}
\cR=\cR^{(0)}\cR^{(1)}\cR^{(2)}\,.
\label{univR}
\end{align}
Here 
\begin{align}
&\cR^{(1)}=\exp\Bigl(\sum_{r\ge1}r\kappa_r \hp_r\otimes \hp_{-r}\Bigr)\,, 
\label{R0}
\end{align}
and 
\begin{align}
\cR^{(2)}=1+\kappa_1\sum_{i\in\Z}\ep_i\otimes \fp_{-i}+\cdots\,
\label{R1}
\end{align}
is the canonical element of $\E_{\gge}^\perp\otimes\E_{\lle}^\perp$.
In \eqref{R1}, $\cdots$ stands for terms whose first component has homogeneous degree $\ge2$. 
The element $\cR^{(0)}$ is formally defined as 
\begin{align*}
&\cR^{(0)}= 
q^{-c\otimes d-d\otimes c-c^\perp\otimes d^\perp-d^\perp\otimes c^\perp}\,,
\\
&C=q^c,C^\perp=q^{c^\perp},D=q^d,D^\perp=q^{d^\perp}\,.  
\end{align*}
The expression $(\cR^{(0)})^{-1}\Delta^{\mathrm{op}}(x)\cR^{(0)}$ has a well defined meaning, 
and the intertwining property \eqref{int-R} should be understood as 
\begin{align*}
\cR^{(1)}\cR^{(2)}\ \Delta(x)=
\left((\cR^{(0)})^{-1}\Delta^{\mathrm{op}}(x)\cR^{(0)}\right)\
\cR^{(1)}\cR^{(2)} \quad (x\in \E).
\end{align*}
The element $\cR^{(0)}$ is well defined on tensor products
of representations which are principally graded and on which 
$c$ acts as $0$. This is the case for all representations considered in this paper.

\subsection{Fock representations}\label{rep sec}
Let $V$ be an $\E'$ module, and let $L,K\in\C^\times$. 
We say that $V$ has  {\it level} $(L,K)$ if 
the central element $C$ acts as the scalar $L$ and $C^\perp$ as $K$.
In this paper we consider only modules of level $(1,K)$. 
Then the operators $h_r$ are mutually commutative on $V$. 
We say that $V$ is {\it quasi-finite} if 
it is graded by the principal degree, $V=\oplus_{n\in\Z} V_n$,  
and  $\dim V_n<\infty$ for all $n$.  
We say it is bounded if $V_n=0$ for $n\ll 0$. 
For an $\E'$ module $V$ and $u\in\C^\times$, 
we denote by $V(u)$ the pullback of $V$ by the automorphism 
\begin{align*}
&s_u: e(z)\mapsto e(z/u),
\quad f(z)\mapsto f(z/u),\quad \psi^\pm(z)\mapsto \psi^\pm(z/u),
\quad C\mapsto C\,,\quad C^\perp\mapsto C^\perp\,.
\end{align*}

Let $\phi(z)$ be a rational function such that 
$\phi(z)$ is regular at $z=0,\infty$ and $\phi(0)\phi(\infty)=1$.
We say that $V$ is a lowest weight module with lowest weight 
$\phi(z)$ if it is generated by a vector $v$ 
which satisfies 
\begin{align*}
f(z)v=0,\quad \psi^{\pm}(z)v=\phi^{\pm}(z)v\,.
\end{align*}
Here $\phi^\pm(z)$ means the expansion of $\phi(z)$ at $z^{\mp1}=0$. 
For each such $\phi(z)$, there exists a unique irreducible lowest weight module 
$L_{\phi(z)}$ with lowest weight $\phi(z)$. 
Assigning degree $0$ to $v$ we have the principal grading 
$L_{\phi(z)}=\oplus_{n=0}^\infty (L_{\phi(z)})_n$, 
and $L_{\phi(z)}$ is quasi-finite \cite{M}. 

\medskip

The most basic lowest weight $\E'$ module is the Fock module. 
For $u\in\C^\times$, the Fock module $\F(u)$ is defined to be 
the irreducible lowest weight $\E'$ module with level $(1,q)$ and lowest weight 
\begin{align}
\phi(u,z)=\frac{q^{-1}-q u/z}{1-u/z}\,.
\label{Fock-wt}
\end{align}
As a vector space, $\F(u)$ has
a basis $\{\ket\lambda\}_{\lambda\in\cP}$ labeled by all partitions.

We use the following convention for partitions. A partition is a sequence of 
non-negative 
integers $\la=(\la_1,\la_2,\cdots)$ such that $\la_i\ge\la_{i+1}$ for all 
$i\ge1$, and $\la_i=0$ for $i$ large enough.
In particular, we write $\emptyset=(0,0,0,\dots)$.
The set of all partitions is denoted by $\cP$. 
The dual partition $\la'$ is given by $\la'_i=\sharp\{j\ |\ \la_j\geq i\}$.
We set $|\la|=\sum_{j\ge1}\la_j$ for $\la\in\cP$. 
For $j\ge1$ and $\lambda\in\cP$ we write
$\lambda+\mathbf{1}_j=(\la_1,\la_2,\cdots,\la_j+1,\cdots)$. 

We call a pair of natural numbers $(x,y)$ {\it convex corner} of $\la$ 
if $\la'_{y+1}<\la'_y=x$,
and {\it concave corner} of $\la$ 
if $\la'_y=x-1$ and in addition  $y=1$ or $\la'_{y-1}>x-1$.
We denote by $CC(\la)$ and $CV(\la)$ the set of concave and 
convex corners of $\la$, respectively.

Then the action of the generators is given as follows \cite{FT}:
\begin{align*}
\bra{\la+{\bf 1}_j}e(z) \ket{\la}
=
\prod_{s=1}^{j-1}  \psi(q_1^{\la_s-\la_j-1}q_3^{s-j})
\prod_{s=1}^{j-1} \psi(q_1^{\la_j-\la_s}q_3^{j-s})
\ \cdot \delta(q_1^{\la_j}q_3^{j-1}u/z)\,,
\end{align*}
\begin{align*}
\bra{\la}f(z) \ket{\la+{\bf 1}_j}
=\frac{q-q^{-1}}{\kappa_1}
\prod_{s=j+1}^{\ell(\lambda)} \psi(q_1^{\la_s-\la_j-1}q_3^{s-j})
\prod_{s=j+1}^{\ell(\lambda)+1}\psi(q_1^{\la_j-\la_s}q_3^{j-s})
\ \cdot \delta(q_1^{\la_j}q_3^{j-1}u/z)\,,
\end{align*}
\begin{align*}
\bra{\la}\psi^{\pm}(z) \ket{\la}=
\hspace{-10pt}\prod_{(i,j)\in\ CV(\la)}
\hspace{-6pt}\psi(q_3^iq_1^jq_2u/z)
\hspace{-10pt}\prod_{(i,j)\in\ CC(\la)} 
\hspace{-6pt}\psi(q_3^iq_1^jq_2^2u/z)^{-1}\,.
\end{align*}
In the above, we set $\psi(z)=(q-q^{-1}z)/(1-z)$ and assume that $\la,\la+\one_j\in\cP$. 
In all other cases the matrix elements are defined to be zero.
In terms of the generators $h_r$, we have for $r\in\Z\backslash\{0\}$
\begin{align}
h_r\ket{\emptyset}=\gamma_r\ket{\emptyset},
\quad \gamma_r=\frac{1-q_2^r}{r\kappa_r}u^r \,.
\label{h-eig}
\end{align}

\medskip

The generators $\hp_r$ act as a Heisenberg algebra on $\F(u)$,  
\begin{align}
[\hp_r,\hp_s]=\frac{q^r-q^{-r}}{r\kappa_r}\delta_{r+s,0}
\quad (r,s\in\Z\backslash\{0\})
\,, 
\label{Heis}
\end{align}
and $\F(u)$ is an irreducible module over this Heisenberg algebra. 
The generators $\ep(z),\fp(z)$ act by vertex operators, 
\begin{align}
&\ep(z)=\frac{1-q_2}{\kappa_1}\,
u \,
\exp\left(\sum_{r=1}^\infty \frac{\kappa_r}{1-q_2^r}\,\hp_{-r}z^r\right)
\exp\left(\sum_{r=1}^\infty \frac{q^r\kappa_r}{1-q_2^r}\,\hp_{r} z^{-r}\right)\,,
\label{VO1}\\
&\fp(z)= \frac{1-q_2^{-1}}{\kappa_1} u^{-1}\,
\exp\left(-\sum_{r=1}^\infty \frac{q^r\kappa_r}{1-q_2^r}\,\hp_{-r} z^r\right)
\exp\left(-\sum_{r=1}^\infty \frac{q^{2r}\kappa_r}{1-q_2^r}\,\hp_{r} z^{-r}\right)\,.
\label{VO2}
\end{align}

\section{Shuffle algebras}\label{shuffle sec}

It is known that the algebra 
$\E_{>}=\langle e_n\ (n\in\Z)\rangle$ 
has a presentation in terms of 
certain algebra of rational functions called the shuffle algebra. 
In this section we introduce an extension 
of the shuffle algebra which gives a functional realization of the Fock 
modules.

\subsection{Algebra $Sh_{0}$}
First, let us recall the definition of the shuffle algebra 
\begin{align*}
Sh_{0}=\oplus_{n=0}^\infty Sh_{0,n}. 
\end{align*}
 
We set $Sh_{0,0}=\C$, $Sh_{0,1}=\C[x^{\pm1}]$. 
For $n\ge2$, $Sh_{0,n}$ is the space of all symmetric rational functions of the form 
\begin{align*}
F(x_1,\cdots,x_n)=\frac{f(x_1,\cdots,x_n)}{\prod_{1\le i<j\le n}(x_i-x_j)^2}, 
\quad f(x_1,\cdots,x_n)\in \C[x_1^{\pm1},\cdots,x_n^{\pm1}]^{\GS_n}, 
\end{align*}
satisfying the {\it wheel condition}
\begin{align}
f(x_1,\cdots,x_n)=0
\quad \text{if $(x_1,x_2,x_3)=(x,q_1x,q_1q_2x)$ or $(x,q_2x,q_1q_2x)$}. 
\label{wheel1}
\end{align}
Note that since $f(x_1,\cdots,x_n)$ is symmetric, from \eqref{wheel1}, we also have $f(x_1,\cdots,x_n)=0$ if
$(x_1,x_2,x_3)=(x,q_ix,q_iq_jx)$ or $(x_1,x_2,x_3)=(x,q_iq_jx,q_ix)$ for $i,j\in\{1,2,3\}$, $i\neq j$.

We define the shuffle product $\ast$ of elements $F\in Sh_{0,m}$ and $G\in Sh_{0,n}$  by the formula 
\begin{align*}
(F\ast G)(x_1,\cdots,x_{m+n})=
\Sym\Bigl[
F(x_1,\cdots,x_m)G(x_{m+1},\cdots,x_{m+n})\prod_{\genfrac{}{}{0pt}{}{1\le i\le m}{ 1\le j\le n}}\omega(x_{m+j},x_i)\Bigr]\,,
\end{align*}
where
\begin{align*}
\omega(x,y)=\frac{(x-q_1y)(x-q_2y)(x-q_3y)}{(x-y)^3}
=\frac{g(x,y)}{(x-y)^3}\,.
\end{align*}
Here and after we set 
\begin{align*}
\Sym f(x_1,\cdots,x_n) =\frac{1}{n!}\sum_{\sigma\in\GS_n} f(x_{\sigma(1)},\cdots,x_{\sigma{(n)}})\,.
\end{align*}
It is easy to check that the space $Sh_{0}$ becomes 
an associative algebra under the product $\ast$. 
The following fact is known (see \cite{SV2}, \cite{FT}, \cite{Ng}). 

\begin{prop}\label{prop:sigma}
Algebra $Sh_{0}$ is generated by the subspace $Sh_{0,1}$.
There is an isomorphism of algebras $\sigma: \ \E_{>}\simeq Sh_{0}$ 
such that
\begin{align*}
 \sigma (e_i)= c_1 x^i\in  Sh_{0,1}, \qquad i\in\Z,
\end{align*}
where $c_1=q_3/((1-q_1)(1-q_3))$. 
\end{prop}
Under the isomorphism above, 
the graded component $(\E_{>})_{n,d}$ 
corresponds to the subspace $(Sh_{0,n})_d$ of $Sh_{0,n}$ 
consisting of functions of homogeneous degree $d\in\Z$.

\subsection{The bimodule $Sh_{1}(u)$}

We fix $u\in \C^{\times}$, and consider a linear space 
\begin{align*}
Sh_{1}(u)=\oplus_{n=0}^\infty Sh_{1,n}(u). 
\end{align*}
We set $Sh_{1,0}(u)=\C$, $Sh_{1,1}(u)=(x-u)^{-1}\C[x^{\pm1}]$. 
For $n\ge 2$, 
$Sh_{1,n}(u)$ is the space of all rational functions of the form
\begin{align*}
&F(x_1,\cdots,x_n)
=\frac{f(x_1,\cdots,x_n)}
{\prod_{1\le i<j\le n}(x_i-x_j)^2
\prod_{i=1}^n(x_i-u)}\,,
\\
&f(x_1,\cdots,x_n)
\in \C[x_1^{\pm1},\cdots,x_n^{\pm1}]^{\GS_n}\,,
\end{align*}
such that they satisfy both the wheel condition \eqref{wheel1} 
and an additional wheel condition 
\begin{align}
&f(u,q_2u,x_3,\cdots,x_n)=0\,.
\label{wheel2}
\end{align}
In what follows, we denote the element $1\in Sh_{1,0}(u)=\C$ by $\one$. 

For $F\in Sh_{0,m}$ and $G\in Sh_{1,n}(u)$, we set 
\begin{align*}
&(F\ast G)(x_1,\cdots,x_{m+n})= 
\Sym\Bigl[
F(x_1,\cdots,x_m)G(x_{m+1},\cdots,x_{m+n})
\hspace{-7pt}\prod_{\genfrac{}{}{0pt}{}{1\le i\le m}{1\le j\le n}}\omega(x_{m+j},x_i)
\prod_{i=1}^m\phi(u,x_i)\Bigr]\,,
\\
&(G\ast F)(x_1,\cdots,x_{m+n})= 
\Sym\Bigl[
G(x_{m+1},\cdots,x_{m+n})F(x_1,\cdots,x_m)\hspace{-7pt}
\prod_{\genfrac{}{}{0pt}{}{1\le i\le m}{1\le j\le n}}\omega(x_i,x_{m+j})
\Bigr]\,,
\end{align*}
where $\phi(u,z)$ is given in \eqref{Fock-wt}.
With this definition, by Proposition \ref{prop:sigma}, 
 $Sh_{1}(u)$ is an $\E_{>}$ bimodule. Later we will prove that $Sh_{1}(u)$ is a cyclic bimodule, and
$\one$ is a cyclic vector (see Corollary \ref{cyclic}). 

We upgrade the left $\E_{>}$ action to 
make $Sh_{1}(u)$ a left $\E'$ module of level $(1,q)$. 
\begin{prop}
The following formula defines a left action of $\E'$ on 
 $Sh_{1}(u)$. 
\begin{align}
&e_k \cdot F= c_1 x^k*F\,,\notag\\
&h_r \cdot F=\bigl(-\frac{1}{r}\sum_{i=1}^n x_i^r
+\gamma_r  
\bigr)F\,, \label{h act}
\\
&f_k \cdot F
=c_2 n \Bigl(\Res_{z=0}+\Res_{z=\infty}\Bigr)
\frac{F(x_1,\cdots,x_{n-1},z)z^k}{\prod_{i=1}^{n-1}\omega(z,x_i)}
\frac{dz}{z}\,.\notag
\end{align}
Here  $F\in Sh_{1,n}(u)$,  
$k\in\Z$, $r\in\Z\backslash\{0\}$, $\gamma_r$ is defined in \eqref{h-eig}, 
$c_1=q_3/((1-q_1)(1-q_3))$ and $c_2=q_3^{-1}/(1-q_2)$.
This action commutes with the right action of $\E_{>}$.  
\end{prop}
\begin{proof}
The proof is done by a direct computation. 
As an example we sketch the verification of the relation
$[e(z),f(w)]=(1/\kappa_1)\delta(z/w)\bigl(\psi^+(z)-\psi^-(z)\bigr)$.
Let $F\in Sh_{1,n}(u)$, $k,l\in\Z$. From the above definition we deduce that 
\begin{align*}
&[e_k,f_l]F=-\frac{1}{\kappa_1} 
\times
\Bigl(\Res_{z=0}+\Res_{z=\infty}\Bigr)\prod_{i=1}^n\frac{\omega(x_i,z)}{\omega(z,x_i)}\cdot \phi(u,z)
\cdot z^{k+l}\frac{dz}{z}\times F
\,.
\end{align*}
Comparing this with the expansions at $z^{\pm1}\to\infty$
\begin{align*}
\frac{\omega(x,z)}{\omega(z,x)}=\exp\Bigl(-\sum_{\pm r>0}\frac{1}{r}\kappa_r x^r z^{-r}\Bigr)\,,
\quad
\phi(u,z)=q^{\mp1}\exp\Bigl(\sum_{\pm r>0}\kappa_r\gamma_r z^{-r}\Bigr)\,,
\end{align*}
we obtain the desired relation. The rest of the relations can be checked similarly. 
In particular, the cubic Serre relations follow from the identity 
\begin{align*}
\mathop{\mathrm{Sym}}_{x_1,x_2,x_3} \frac{x_2}{x_3}\bigl(\omega_{3,1}\omega_{3,2}\omega_{2,1}
-\omega_{3,1}\omega_{2,3}\omega_{2,1}-\omega_{1,3}\omega_{1,2}\omega_{3,2}
+\omega_{1,2}\omega_{1,3}\omega_{2,3}\bigr)=0\,,
\end{align*}
where $\omega_{i,j}=\omega(x_i,x_j)$.
\end{proof}

\subsection{Functional realization of Fock module}
We consider the following left $\E'$ submodule of $Sh_{1}(u)$, 
\begin{align}
\J_0=\mathrm{Span}_{\C}
\{G\ast F\mid G\in Sh_{1}(u), F\in Sh_{0,n}, n\ge 1\}\subset Sh_{1}(u)\,.
\label{J0}
\end{align}
The following gives a realization of the Fock module 
as a quotient of a space of rational functions. 
\begin{prop}\label{fock isom}
 We have the isomorphism of left $\E'$ modules $Sh_{1}(u)/\J_0\simeq \F(u)$.
\end{prop}
\begin{proof}
The module  $Sh_{1}(u)/\J_0$ contains the lowest weight vector
$\one$ with the same lowest weight \eqref{Fock-wt} as the Fock module. 
Hence, in order to prove the isomorphism, it is sufficient to show that 
each of its graded component has the same dimension as that of the Fock
module. We show this in Appendix \ref{gordon sec}, Corollary \ref{Gordon-cor}.
\end{proof}

Therefore we have the canonical projection map $\pi$.
\begin{cor}
There exists a unique surjective homomorphism of left $\E'$ modules
\begin{align}\label{pi}
\pi:\,  Sh_{1}(u) \to \F(u), \qquad \one\to \ket{\emptyset} 
\end{align}
which factorizes through $Sh_{1}(u) /J_0$.
\end{cor}

\subsection{The subspace $N$}\label{L sec}
We have a short exact sequence of left $\E'$ modules
\begin{align*}
0\to \J_0 \to Sh_{1}(u) \to \F(u)\to 0. 
\end{align*}
In this section, we split this sequence in the category of
vector spaces. The reason for the choice of this particular splitting
will be clarified later (see \eqref{LL1}).

Define a linear map $\kappa:\ \F(u) \to Sh_{1}(u)$ by the requirements $\kappa(\ket{\emptyset})=\one$ and
\begin{align}\label{sh q com}
\kappa(h_{-r}^\perp (v))=h_{-r}^\perp \kappa(v) -q^r \kappa(v) h^\perp_{-r},  
\end{align}
for all $r>0$ and $v\in\F(u)$. Here we use the bimodule action
of $h_{-r}^\perp\in\E_>$.

Since $\F(u)$ is cyclic with respect to the
algebra generated by $h_{-r}^\perp$, $r>0$, the map $\kappa$
is uniquely defined.
We clearly have $\pi \kappa = id$ and in particular, $\kappa$ is  
injective. Let
\begin{align}\label{subspace N}
N=\kappa(\F(u))\subset Sh_{1}(u).
\end{align}
We clearly have a direct sum of vector spaces
\begin{align}\label{L plus}
Sh_{1}(u)=\J_0\oplus N. 
\end{align}

Now we have the following.
\begin{cor}\label{cyclic}
The space $Sh_{1}(u)$ is a cyclic $\E_>$ bimodule with cyclic vector $\one$. It is a free $\E_>$ right module generated by $N$. It is also a free 
$\E_>$ left module generated by $N$. 
\end{cor}
\begin{proof}
The first statement follows from Proposition \ref{fock isom}. 
The right and left actions of $Sh_{0}$ are clearly free. The corollary follows.
\end{proof}

\medskip

The subspace $N$ has a curious description in terms of regularity conditions. 

We call a function $G(x_1,\dots,x_n)\in Sh_{1}(u)$ regular at zero if there exists a well-defined limit $\lim_{t\to 0}G(ty_1,\dots,ty_k, x_{k+1},\dots,x_n)$, $k=1,\dots,n$. We call a function $G(x_1,\dots,x_n)\in Sh_{1}(u)$ 
regular at infinity if there exists a well-defined limit $\lim_{t\to \infty }G(ty_1,\dots,ty_k, x_{k+1},\dots,x_n)$, $k=1,\dots,n$.

\begin{prop}
A function $G(x_1,\dots,x_n)\in Sh_{1}(u)$ belongs to $N$ if and only if it is regular at zero, regular at infinity and
$\lim_{t\to 0}G(ty_1,\dots,ty_n)=0$.
\end{prop}
\begin{proof} It is known \cite{Ng}, that a function $F(x_1,\dots,x_m)\in 
Sh_{0,m}$ belongs to the commutative algebra generated by 
$h_{-r}^\perp$, $r>0$, 
if and only if it is regular at zero and at infinity. Then it is easy to check
that action (\ref{sh q com}) preserves the regularity and vanishing conditions described in the proposition. 
It is easy to check that 
action (\ref{sh q com}) preserves the regularity conditions 
at zero and at infinity.
Noting $\lim_{t\to 0}\phi(u,tx)=q$
we see further that the vanishing condition is also preserved.
Since $\one$ satisfies these conditions, we obtain the only if part. For the if part, we compute the dimension of the space of functions using the same filtration as in \cite{Ng}. 
\end{proof}

We remark that if one defined a map $\tilde\kappa:\ \F(u) \to Sh_{1}(u)$ by changing $q$ to $q^{-1}$ in 
\eqref{sh q com} then the image of $\tilde\kappa$ would consist of functions $G(x_1,\dots,x_n)\in Sh_{1}(u)$  which are regular at zero, regular at infinity and satisfy $\lim_{t\to \infty }G(ty_1,\dots,ty_n)=0$.

\section{The subspace of matrix elements of $L$ operators}
\label{L oper sec}
In this section we construct an inclusion of bimodule $Sh_{1}(u)$ 
to a completion of algebra $\E_{\gge}$. 
Under this inclusion the subspace 
$N \subset Sh_{1}(u)$, see \eqref{subspace N},
has a description in terms of matrix elements of $L$ operators.

\subsection{The matrix elements of $L$ operators}\label{L operator subset}
Let $\cR$ be the universal $R$ matrix \eqref{univR}, 
and set $\cR'=q^{c\otimes d+d\otimes c}\cR$.  
For bounded quasi-finite modules $V,W$, 
$\cR'$ gives a well defined operator on a tensor product 
$V(u_1)\otimes W(u_2)$ for generic $u_1,u_2$.

For $v\in\F(u)$ and $w\in\F(u)^*$, let
\begin{align*}
L_{w,v}=(1\otimes w)\cR'(1\otimes v)
\end{align*}
denote the matrix element of $\cR'$
with respect to the second component. We call elements $L_{w,v}$ 
matrix elements of $L$ operators. 

We are mostly concerned with the case $w=\bra{\emptyset}$. 
In what follows we abbreviate $\bra{\emptyset}$, $\ket{\emptyset}$ simply as
$\emptyset$ in the index of the matrix elements of $L$ operators.

If a coproduct of an element of $\E$ is known, one can compute its 
commutation relations with the matrix elements of $L$ operators. 
In particular, we have the following commutation relations with perpendicular generators which involve only matrix elements of $L$ operators 
with $w=\bra{\emptyset}$.

\begin{lem}\label{LL-comm}
For all $r,n>0$ and $v\in \F(u)$, we have 
\begin{align}
&[h^\perp_{-r},L_{\emptyset,v}]_{q^r}= L_{\emptyset,h^\perp_{-r}v}\,,
\label{LL1}\\
&[e^\perp_{-n},L_{\emptyset,v}]_{q^{-n}}
= L_{\emptyset,e^\perp_{-n}v}+
q^{-n}\sum_{j\ge1}L_{\emptyset,\psi^{+,\perp}_{j}v}\cdot e^\perp_{-n-j}\,. 
\label{LL2}
\end{align}
In addition we have 
\begin{align}
&[e^\perp_{0},L_{\emptyset,v}]
= L_{\emptyset, e^\perp_0v}
-\gamma_1L_{\emptyset, v}
+
\sum_{j\ge1}L_{\emptyset,\psi^{+,\perp}_{j}v}\cdot e^\perp_{-j}\,,
\label{LL3}\\
&[f^\perp_{0},L_{\emptyset,v}]
= L_{\emptyset, f^\perp_0v}
-\gamma_{-1}L_{\emptyset, v}
+
\sum_{j\ge1}L_{\emptyset,f^{\perp}_{j}v}\cdot \psi^{-,\perp}_{-j}\,.
\label{LL4}
\end{align}
Here we set $[A,B]_p=AB-pBA$.
\end{lem}
\begin{proof}
The element $\cR'$ has the intertwining property
\begin{align*}
&\cR'\left(h^\perp_r\otimes 1+\bigl(C^\perp\bigr)^{-r}\otimes h^\perp_{r}\right)
=\left(q^{-r} h^\perp_r\otimes 1+1\otimes h^\perp_{r}\right)\cR'\,,
\\
&\cR'\left(q^rh^\perp_{-r}\otimes 1+1\otimes h^\perp_{-r}\right)
=\left(h^\perp_{-r}\otimes 1+\bigl(C^\perp\bigr)^{r}\otimes h^\perp_{-r}\right)\cR'\,,
\\
&\cR'
\Bigl(
q^n e_n^\perp\otimes 1+1\otimes e_n^\perp
+q^n\sum_{j\ge1}e^{\perp}_{n-j}\otimes \psi^{+,\perp}_j\Bigr)\\
&\hspace{50pt}=\Bigl(
e_n^\perp\otimes 1+\bigl(C^\perp\bigr)^{n}\otimes e_n^\perp
+\sum_{j\ge1} \bigl(C^\perp\bigr)^{n}\psi^{+,\perp}_j\otimes e^{\perp}_{n-j}
\Bigr)
\cR'\,,
\\
&\cR'\Bigl(
f_n^\perp\otimes 1+\bigl(C^\perp\bigr)^{n}\otimes f_n^\perp
+\sum_{j\ge1} \bigl(C^\perp\bigr)^{n}\psi^{-,\perp}_{-j}\otimes f^{\perp}_{n+j}
\Bigr)
\\
&\hspace{50pt}=\Bigl(
q^n f_n^\perp\otimes 1+ 1\otimes f_n^\perp
+q^n\sum_{j\ge1}f^{\perp}_{n+j}\otimes \psi^{-,\perp}_{-j}
\Bigr)\cR'\,,
\end{align*}
where $r>0$ and $n\in\Z$.
Taking the matrix element between $\bra{\emptyset}$ and $v$ 
in the second component, we obtain the lemma. 
\end{proof}

The above intertwining relations allow us to compute $L_{\emptyset,\emptyset}$ explicitly.
\begin{prop}\label{L-comm}
The element $L_{\emptyset,\emptyset}$ satisfies 
the commutation relations 
\begin{align}
&(z-u)e(z)L_{\emptyset,\emptyset}=(q^{-1}z-qu )L_{\emptyset,\emptyset}e(z)\,,
\label{e-L}\\
&(q^{-1}z-qu)f(z)L_{\emptyset,\emptyset}=(z-u )L_{\emptyset,\emptyset}f(z)\,,
\label{f-L}\\
&[h_r,L_{\emptyset,\emptyset}]=0\quad (\forall r\neq0).
\label{h-L}
\end{align}
Explicitly we have 
\begin{align}
L_{\emptyset,\emptyset}=q^{-d^\perp}\exp\left(\sum_{r=1}^\infty (1-q_2^{-r})h_r u^{-r}\right).
\label{Lee2}
\end{align}
\end{prop}

\begin{proof}
In Lemma \ref{LL-comm}, we consider the case $v=\ket{\emptyset}$. 
We need two more formulas derived similarly:
\begin{align*}
&[h_1^\perp,L_{\emptyset,\emptyset}]_q
=-q L_{\bra{\emptyset}h^\perp_1,\emptyset}\,,\\
& [e_1^\perp,L_{\emptyset,\emptyset}]_q
=-qu C^\perp L_{\bra{\emptyset}h^\perp_1,\emptyset}
-(1-q^2)u h^\perp_1L_{\emptyset,\emptyset}\,.
\end{align*}

Using
\begin{align*}
&e_0^\perp=h_1,\quad e_1^\perp=C^\perp f_1,\quad e_{-1}^\perp=e_1,\quad h_1^\perp=f_0,
\\
&f_0^\perp=h_{-1},\quad f_1^\perp= f_{-1},\quad f_{-1}^\perp=(C^\perp)^{-1}e_{-1},\quad h_{-1}^\perp=e_0\,,
\end{align*}
and the relations
\begin{align*}
&e^\perp_{-1}\ket{\emptyset}=u h^\perp_{-1}\ket{\emptyset}\,,
\quad 
f^\perp_{-1}\ket{\emptyset}=(qu)^{-1} h^\perp_{-1}\ket{\emptyset}\,,
\\
&\bra{\emptyset}e^\perp_{1}=q u\bra{\emptyset} h^\perp_{1}\,,
\quad 
\bra{\emptyset}f^\perp_{1}=u^{-1} \bra{\emptyset}h^\perp_{1}\,,
\end{align*}
which follow from \eqref{VO1}--\eqref{VO2}, 
we find
\begin{align*}
&[h_1,L_{\emptyset,\emptyset}]=[h_{-1},L_{\emptyset,\emptyset}]=0,\\
&(e_1-u e_0)L_{\emptyset,\emptyset}=L_{\emptyset,\emptyset}(q^{-1}e_1-qu e_0),\\
&(q^{-1}f_1-q u f_0)L_{\emptyset,\emptyset}=L_{\emptyset,\emptyset}(f_1-u f_0)\,.
\end{align*}
Taking commutators between the last two lines and  $h_{\pm1}$, we obtain 
\eqref{e-L} and  \eqref{f-L}.

Furthermore, \eqref{e-L} and  \eqref{f-L} imply that 
\begin{align*}
(z-u)(q^{-1}z-qu)\delta(z/w)[\psi^{+,\perp}(z)-\psi^{-,\perp}(z),L_{\emptyset,\emptyset}]=0\,.
\end{align*}
Let $[\psi^{+,\perp}(z)-\psi^{-,\perp}(z),L_{\emptyset,\emptyset}]=\sum_{j\in\Z}X_jz^{-j}$. Then
\begin{align*}
q^{-1}X_{j+1}-(q+q^{-1})u X_j+qu^2X_{j-1}=0\,.
\end{align*}
We have $X_0=0$, and we already know that $X_{\pm1}=0$. 
From this follows  \eqref{h-L}.

The unique element in the competion of $\E^\perp_{\gge}$ with respect to the homogeneous degree $\mathrm{hdeg}$ when it becomes large satisfying \eqref{e-L}--\eqref{h-L}
is given by \eqref{Lee2}. The lemma follows.
\end{proof}

We denote by $\hat\E_{\gge}$ the completion of algebra $\E_{\gge}$
with respect to the homogeneous degree $\mathrm{hdeg}$ when it becomes large. 
 
\begin{cor}\label{cor:Lv}
 We have $L_{\emptyset,v}\in \hat\E_{\gge}$ for all $v\in \F(u)$. \end{cor}
\begin{proof}
We have $L_{\emptyset,\emptyset}\in \hat\E_{\gge}$ from \eqref{Lee2}. 
Since $h_{-r}^\perp\in\E_{>}$ for $r>0$,
the corollary follows
from  \eqref{LL1}.
\end{proof}

We denote $\mc N$ the space of matrix elements of
$L$ operators with the
first component $\bra{\emptyset}$: 
\begin{align*}
\mc N=\on{Span}_\C\{  L_{\emptyset,v}\mid v\in \F(u)\}\subset \hat
\E_{\gge}.
\end{align*}

\subsection{Inclusion of the shuffle algebra to $\hat\E_{\gge}$}\label{Inclu}
Consider the  $\E_{>}$ bimodule 
\begin{align*}
\mc S(u)=\E_{>}\cdot L_{\emptyset,\emptyset}\cdot \E_{>}\subset\hat\E_{\gge}\,.
\end{align*}
Since  $L_{\emptyset,\emptyset}$ satisfies relations 
\eqref{e-L}, there is a surjective map of 
$\E_{>}$ bimodules
\begin{align}
\iota:\, Sh_{1}(u)\to \mc S(u)
\,,\qquad  
\one\mapsto L_{\emptyset,\emptyset}\,.
\label{proj}
\end{align}

\begin{lem}\label{iso-L}
The map  $\iota$ in \eqref{proj} is an isomorphism. 
\end{lem}
\begin{proof}
It suffices to show that $\iota$ is injective. 
Suppose that $G=\sum_j F_j\ast \one \ast H_j$ 
is in the kernel of \eqref{proj} where $F_j,H_j\in Sh_{0}$. 
Let $a_j,b_j\in \E_{>}$
be the elements corresponding to $F_j,H_j$.  
In the completion of $\E_{\gge}$
we have the commutation relation
\begin{align*}
e_n L_{\emptyset,\emptyset}=L_{\emptyset,\emptyset} \tilde{e}_n,
\quad 
\tilde{e}_n=
q e_n+(q-q^{-1})\sum_{j\ge1}u^{-j}e_{n+j}\,.
\end{align*}
Using this we move the $a_j$'s to the right and obtain
\begin{align*}
0=\sum_j a_jL_{\emptyset,\emptyset}b_j=
\sum_j L_{\emptyset,\emptyset}\tilde{a}_jb_j\,,
\end{align*}
where $\tilde{a}_j$ is obtained from $a_j$ by substituting 
$e_n$ by $\tilde{e}_n$.
Since $L_{\emptyset,\emptyset}$ is invertible, this implies that 
$\sum_j \tilde{a}_jb_j=0$. 
We may assume that $G$ has principal degree, say, $m$. 
Let $\tilde{G}$ be the element of the completion of $Sh_{0}$ 
corresponding to $\sum_j \tilde{a}_jb_j$.  
Then we observe that $\prod_{1\le i<j\le m}(x_i-x_j)^2\tilde{G}$ is nothing but the expansion of 
the rational function $\prod_{1\le i<j\le m}(x_i-x_j)^2 G$ at $x_1=\cdots=x_m=0$.
Hence we have $G=0$. 
\end{proof}

Then we have the identifications of spaces $N$, $\mc N$ and $\F(u)$.
\begin{lem}\label{L ident}
We have
$\iota(N)=\mc N\subset \mc S(u).$
For any $v\in\F(u)$, we have
\begin{align*}
\pi\iota^{-1}(L_{\emptyset,v})=v\in \F(u).
\end{align*}
\end{lem}
\begin{proof}
We have $\pi\iota^{-1} (L_{\emptyset,\emptyset})=\ket{\emptyset}\in \F(u)$.
The module 
$\F(u)$ is cyclic with respect to the action of 
the $h^{\perp}_{-r}$'s.
Formula \eqref{LL1} implies that 
$h^\perp_{-r}L_{\emptyset,v}\equiv L_{\emptyset,h^\perp_{-r}v}$   
holds modulo right action.
Since the map $\iota$ is $\E_{>}$ linear and $h^\perp_{-r}\in \E_{>}$,  
 we obtain the assertion. 
\end{proof}

We capture various maps on Figure \ref{Fig 3}.

\begin{figure}
\begin{picture}(300,00)(-80,30)
\put(0,0){$Sh_1(u)$}
\put(100,0){$\mc S(u)\subset\hat\E_\gge$}
\put(13,-17){\rotatebox{90}{$\subset$}}
\put(11,-30){$N$}
\put(12,-43){\rotatebox{90}{$\in$}}
\put(12,-55){$\one$}
\put(107,-17){\rotatebox{90}{$\subset$}}
\put(105,-30){$\mc N$}
\put(107,-43){\rotatebox{90}{$\in$}}
\put(103,-56){$L_{\emptyset,\emptyset}$}
\put(61,-82){$\F(u)$}
\put(67,-97){\rotatebox{90}{$\in$}}
\put(65,-112){$\ket\emptyset$}
\put(-52,0){$J_p\oplus N\simeq$}
\put(-30,-8){\vector(2,-1){35}}
\put(33,-8){\vector(1,-2){30}}
\put(55,-75){\vector(-2,3){28}}
\put(44,3){\vector(1,0){45}}
\put(33,-27){\vector(1,0){64}}
\put(62,9){$\iota$}
\put(59,-7){$\simeq$}
\put(59,-37){$\simeq$}
\put(30,-57){$\kappa$}
\put(35,-67){$\simeq$}
\put(-22,-25){$\pi_p$}
\put(40,-45){$\pi$}
\end{picture}
\vskip150pt
\caption{Subspaces isomorphic to Fock space}\label{Fig 3}
\end{figure}

\section{Bethe ansatz}\label{Bethe sec}
\subsection{Integrals of motion}
Let $W$ be a bounded quasi-finite module. We have $W=\oplus_{n\in\Z}W_n$, and
$d^\perp|_{V_n}=n$.
Fixing a parameter $p\in \C^\times$, consider the weighted trace 
\begin{align*}
T_{W}(u;p)=\Tr_{W(u),2} \left(p^{1\otimes d^\perp} \cR'\right)
=\sum_{n\in\Z} (p q^{-c^\perp})^n \,
\Tr_{W(u)_n,2} \left(q^{-d^\perp\otimes c^\perp}\cR^{(1)}\cR^{(2)}\right)\,
,
\end{align*}
where $\cR'$ is defined in Subsection \ref{L operator subset},
and $\Tr_{W(u),2}$ signifies 
the trace in the second tensor component.

Note that the first tensor component of $\cR^{(1)}$ \eqref{R0}
has the homogeneous degree $0$, and that of $\cR^{(2)}$ \eqref{R1}
has non-negative homogeneous degree. Therefore, the operator
$T_{W}(u;p)$ has the form $\sum_{l=0}^\infty T_{W,l}(p)u^{-l}$,
where $T_{W,l}(p)$ is an operator on bounded quasi-finite modules and $\deg T_{W,l}(p)=(0,l)$. 

We introduce the integrals of motion $\{I_{W,l}(p)\}_{n=1}^\infty$ by 
\begin{align*}
\log \left(T_{W,0}(p)^{-1}T_{W}(u;p)\right)
=\sum_{n=1}^\infty I_{W,l}(p) u^{-l}\,.
\end{align*}

For fixed $p$, the integrals of motions form a commutative family:
\begin{align*}
[T_{W_1,l_1}(p),T_{W_2,l_2}(p)]=0,
\end{align*}
for all $l_1,l_2\in\Z_{>0}$ and all bounded quasi-finite modules $W_1,W_2$.

\medskip

When $W(u)$ is the Fock module $\F(u)$, we have the following expression
for $I_{\F,1}(p)$. 

\begin{lem}\label{int expl}
Set $\tilde{p}=pq^{-c^\perp}$. Then the operator $I_{\F,1}(p)$ is given by 
\begin{align}
&I_{\F,1}(p)=k \cdot \tilde{e}^\perp_0(p),\quad
k=\frac{(q_1q_3,\tilde{p};\tilde{p})_\infty}
{(\tilde{p}q_1,\tilde{p}q_3;\tilde{p})_\infty}\,,
\label{first-int}
\end{align}
where 
$\tilde{e}^\perp_0(p)$ is the coefficient of $z^0$ of the twisted current
\begin{align}
e^\perp(z;p)=
e^\perp(z)\prod_{j=1}^\infty\psi^{+,\perp}\bigl(\tilde{p}^{-j}q^{-1}z\bigr)\,,
\label{e0p}
\end{align}
and $(a_1,\cdots,a_m;\tilde{p})_\infty=\prod_{j=1}^m(a_j;\tilde{p})_\infty$, 
$(a;\tilde{p})_\infty=\prod_{k=0}^\infty(1-a \tilde{p}^k)$. 
\end{lem}
\begin{proof}
From the definition of $T_{V}(u,p)$ along with \eqref{R0}, \eqref{R1}, 
we obtain 
\begin{align*}
&T_{\F,0}(p)=
q^{-d^\perp}\Tr_{\F(u),2} \left[\tilde{p}^{1\otimes d^\perp}
\cR^{(1)}\right]\,,
\\
&T_{\F,1}(p)=q^{-d^\perp}\kappa_1
\res_{z=0}\Tr_{\F(u),2} \left[\tilde{p}^{1\otimes d^\perp}
\cR^{(1)}\ \cdot 1\otimes 
f^\perp(z)\right] e^\perp(z)\frac{dz}{z}\,.
\end{align*}
We then substitute \eqref{VO2} for $f^\perp(z)$.
The trace can be calculated by using 
\begin{align*}
&\Tr_{\F(u),2}\left[\tilde{p}^{1\otimes d^\perp}
\exp\Bigl(\sum_{r=1}^\infty A_r1\otimes h^\perp_{-r}\Bigr) 
\exp\Bigl(\sum_{r=1}^\infty B_r1\otimes h^\perp_{r}\Bigr)\right]
\\
&=\frac{1}{(\tilde{p};\tilde{p})_\infty}
\exp\left(\sum_{r=1}^\infty A_rB_r\frac{\tilde{p}^r}{1-\tilde{p}^r}
\frac{q^r-q^{-r}}{r\kappa_r}\right)\,,
\end{align*}
where we set 
\begin{align*}
&A_r=r\kappa_r h^\perp_r\otimes 1-\frac{q^r\kappa_r}{1-q^{2r}}z^r\,,
\quad B_r=-\frac{q^{2r}\kappa_r}{1-q^{2r}}z^{-r}\,.
\end{align*}
Note that
\begin{align*}
\exp\Bigl(\sum_{r=1}^\infty \kappa_r\frac{\tilde{p}^rq^r}{1-\tilde{p}^r}
h^\perp_r z^{-r}\Bigr) 
=\prod_{j\ge1}\psi^{+,\perp}\bigl(\tilde{p}^{-j}q^{-1}z\bigr)\,.
\end{align*}
After simplification we find
\begin{align*}
&T_{\F,0}=q^{-d^\perp}\frac{1}{(\tilde{p};\tilde{p})_\infty}\,,
\quad
T_{\F,1}=u^{-1}q^{-d^\perp}\frac{(q_2^{-1};\tilde{p})_\infty}
{(\tilde{p}q_1,\tilde{p}q_3;\tilde{p})_\infty}
\tilde{e}_0^\perp(p)\,.
\end{align*}
The lemma follows. 
\end{proof}

More generally, if $W$ is a tensor product of several 
Fock modules,  
the operators $\{I_{W,n}(p)\}_{n=1}^\infty$ are 
closely related to the commutative family of operators 
introduced and studied in \cite{FKSW}, \cite{FKSW2}, \cite{KS}.

\subsection{Action of $\tilde e_0^\perp(p)$ on Fock module as a projection}
In this subsection we identify the action of the integral of motion
$\tilde e_0^\perp(p)$ on the Fock module with a projection of operator
$h_1$ acting in $Sh_{1}(u)$ to $N$, along the space $\J_p$.

We fix $p\in \C^{\times}$ and consider the subspace of $Sh_{1}(u)$
\begin{align*}
&\J_p=\mathrm{Span}_{\C}
\{G\ast F-p^nF\ast G
\mid G\in Sh_{1}(u), F\in Sh_{0,n}, n\ge 1\}\subset Sh_{1}(u)\,.
\end{align*}
Unlike \eqref{J0}, it is not an $Sh_{0}$ submodule. 
However, it is clearly preserved by the action of $h_r$, see
\eqref{h act}.
From (\ref{L plus}), we see that for generic $p$,
we have a direct sum of vector spaces
\begin{align}\label{L p plus}
Sh_{1}(u)=\J_p\oplus N.
\end{align}
Denote $\pi_p: \ Sh_{1}(u) \to N$ the projection operator in (\ref{L p plus}) along the first summand.

Recall that for $G(x_1,\dots,x_n)\in Sh_{1,n}(u)$, the action of $h_1$ is simply given by
\begin{align}\label{h1 act}
h_1 G(x_1,\dots,x_n)=
(-\sum_{i=1}^n x_i+\gamma_1)\,G(x_1,\dots,x_n),
\end{align}
see (\ref{h act}). The crucial observation is that the projection of this simple operator to $N$
along $\J_p$ produces the desired integral of motion $\tilde e_0^\perp(p)$.

\begin{thm}\label{main} Under the identification of $N$ and $\F(u)$, we have $\pi_p h_1=\tilde e_0^\perp(p)$.
In other words, for any $v\in\F(u)$ we have
\begin{align*}
\tilde e_0^\perp(p) v=(\kappa^{-1}\circ \pi_p) (h_1 \kappa(v)).
\end{align*}
\end{thm}
\begin{proof}
We use the isomorphism $\iota$, see \eqref{proj}
and Lemma \ref{L ident}. We 
work in $\mc S(u)\subset \hat\E_\gge$ and
make use of the matrix elements of $L$ operators to compute the projection.

By \eqref{LL2} and \eqref{LL3} we have
\begin{align*}
[e_0^\perp, L_{\emptyset,v}]+\gamma_1 
L_{\emptyset,v}
&=L_{\emptyset,e_0^\perp v}
+\sum_{j\ge1}L_{\emptyset,\psi^{+,\perp}_j v}\cdot e^\perp_{-j}
\\
&\equiv L_{\emptyset,e_0^\perp v}
+\sum_{j\ge1}p^j e^\perp_{-j}\cdot L_{\emptyset,\psi^{+,\perp}_j
v}\quad \bmod \J_p\,,
\end{align*}
and for $n>0$
\begin{align*}
e_{-n}^\perp L_{\emptyset,v}
&=L_{\emptyset,e_{-n}^\perp v}
+q^{-n}\sum_{j\ge 0}L_{\emptyset,\psi^{+,\perp}_j v}\cdot e^\perp_{-j-n}
\\
&\equiv L_{\emptyset,e_{-n}^\perp v}
+\sum_{j\ge 0}q^{-n}p^{j+n} e^\perp_{-j-n}\cdot 
L_{\emptyset,\psi^{+,\perp}_j v}\quad \bmod \J_p\,.
\end{align*}
Iterating the latter, we obtain 
\begin{align*}
h_1 L_{\emptyset,v}
&=
[e_0^\perp, L_{\emptyset,v}]+\gamma_1
L_{\emptyset,v}
\\
& 
\equiv L_{\emptyset,e_0^\perp v}+\hspace{-15pt}\sum_{\genfrac{}{}{0pt}{}{k\ge1}{j_k,\cdots,j_2\ge0,j_1\ge1}}\hspace{-15pt}q^{j_k+\cdots+j_1}(pq^{-1})^{j_k}\cdots(pq^{-1})^{kj_1}
L_{\emptyset,e^\perp_{-j_k-\cdots-j_1}\psi^{+,\perp}_{j_k}\cdots\psi^{+,\perp}_{j_1}v}
\quad  \bmod \J_p\,.\\
&=L_{\emptyset,\tilde{e}^\perp_0(p)v}\,,
\end{align*}
where we used definition \eqref{e0p}. 
\end{proof}

\subsection{Bethe ansatz}
Theorem \ref{main} immediately leads to Bethe ansatz statements for the dual module.

Given a quasi-finite left $\E'$ module $V=\oplus_{n=0}^\infty V_n$, 
we consider the graded dual space $V^*=\oplus_{n=0}^\infty\on{Hom}(V_n,\C)$. 
As usual, $V^*$ is a right $\E'$ module with action given by $g f(v)=f(gv)$, $g\in\E'$, $v\in V$, $f\in V^*$. 
Note that the spectrum of any operator and of the dual operator coincide. 
Note also that the dual to a lowest weight Fock left module is a highest weight Fock right module.

For a point $a=(a_1,\cdots,a_n)\in \C^n$ such that $a_i\neq a_j$ ($i\neq j$) and $a_i\neq u$, 
we denote by $ev_a$ the evaluation map $ev_a:\,Sh_{1}(u)\to \C$ defined by
\begin{align*}
&ev_a(F(x_1,\dots,x_n))= F(a_1,\cdots,a_n), \qquad F(x_1,\dots,x_n)\in Sh_{1,n}(u)\,,
\end{align*}
and $ev_a (Sh_{1,m}(u))=0$ for $m\neq n$. 

\begin{lem}\label{Bethe}
We have $ev_a (\J_p)=0$ if and only if $a=(a_1,\cdots,a_n)$ satisfies the {\it Bethe equation} 
\begin{align}
1=q^{-1}p\cdot \frac{a_i-q_2 u}{a_i-u}
\prod_{j(\neq i)}\frac{(a_j-q_1a_i)(a_j-q_2a_i)(a_j-q_3a_i)}
{(a_j-q_1^{-1} a_i)(a_j-q_2^{-1} a_i)(a_j-q_3^{-1} a_i)}\,,
\quad i=1,\cdots,n\,.
\label{BAE}
\end{align}
\end{lem}
\begin{proof}
Let $F\in Sh_{0,m}$, $G\in Sh_{1,k}(u)$ and $m\ge1$, 
$m+k=n$. Then $p^mF*G$ and $G*F$ by definition are symmetrization 
and can be compared term-wise so that the substitutions of variables match. 
These terms become equal if and only if (\ref{BAE}) holds.
\end{proof}

\begin{thm}\label{spectrum}
Let $a=(a_1,\cdots,a_n)$ 
be a solution of \eqref{BAE} such that $ev_a$ is non-zero. 
Then the restriction of $ev_a$ to $N=\kappa(\F(u))$ is an eigenvector in $\F(u)^*$ of the first integral 
of motion $\tilde{e}^\perp_0(p)$ \eqref{e0p} with the eigenvalue
\begin{align*}
E_1(a)=-\sum_{i=1}^n a_i+\gamma_1,
\end{align*}
where we recall that $\gamma_1=u/((1-q_1)(1-q_3))$.
For generic $p$, $ev_a$ is a joint eigenvector of 
$\{I_{\F,n}(p)\}_{n=1}^\infty$. 
\end{thm}
\begin{proof}
When $p=0$, the operator $I_{\F,1}(p)$ on $\F(u)$ has simple spectra. 
Hence it is enough to show that $ev_a$ is an 
eigenvector of $\tilde{e}^\perp_0(p)$ with eigenvalue $E_1(a)$. 

We simply have for any $G\in N$
\begin{align*}
\tilde{e}^\perp_0(p) ev_a(G)=ev_a(\tilde{e}^\perp_0(p)G)=ev_a(h_1G)=E_1(a) ev_a(G),
\end{align*}
where the second equality follows from Theorem \ref{main} together with Lemma \ref{Bethe} and the third equality follows from (\ref{h1 act}).
\end{proof}

\subsection{The off-shell Bethe vector}
For us an off-shell Bethe vector is a vector depending on parameters $a_i$ such that if  $a_i$ satisfy the Bethe equation, it becomes an eigenvector of Hamiltonians. However, such a requirement does not determine it uniquely.

An off-shell Bethe vector is obviously given by the formula 
$(\id\otimes ev_a \circ\kappa)K)$, 
where 
\begin{align*}
K=\sum_{\lambda}
\frac{\bra{\emptyset}h^\perp_{\lambda}\otimes h^\perp_{-\lambda}\ket{\emptyset}}
{\bra{\emptyset}h^\perp_{\lambda}h^\perp_{-\lambda}\ket{\emptyset}}
\end{align*}
is the canonical element 
of the space $\F^*(u)\otimes\F(u)$.
Here and after, for a partition $\la=(\la_1,\cdots,\la_{\ell(\la)})$
we use the notation 
$h^\perp_\la=h^\perp_{\la_1}\cdots h^\perp_{\la_{\ell(\la)}}$, etc..
  
More generally, the off-shell Bethe vector is given by $(\id\otimes ev_a)(K_p+(1\otimes\kappa)K)$ where $K_p\in\F(u)^*\otimes \J_p$. Our goal is to give an explicit formula in terms of partitions for a suitable choice of $K_p$.

Let $\La$ denote the space of symmetric functions. 
For $\la\in\cP$, let $p_\la,m_\la\in\La$ be the power sum and the monomial symmetric functions, 
respectively. 
Using the rescaled generators
\begin{align*}
&\tilde{h}^\perp_r 
=r(1-q_1^r)q_2^{r/2}q_3^rh^\perp_r ,
\end{align*}
we identify the algebra generated by $h_r^\perp$, $r>0$ with $\La$ by
\begin{align*}
\nu^*:\Lambda \overset{\sim}{\longrightarrow}
\C[h^\perp_r]_{r>0},\quad  p_\lambda\mapsto \tilde{h}^\perp_\lambda.
\end{align*}
Introduce further the elements of $Sh_{0}$
\begin{align}
&\epsilon^{(q_3)}_{\lambda}=\epsilon^{(q_3)}_{\lambda_1}\ast\cdots\ast
\epsilon^{(q_3)}_{\lambda_{\ell(\lambda)}},
\quad 
\epsilon^{(q_3)}_{n}(x)=\prod_{i<j}\frac{(x_i-q_3x_j)(x_i-q_3^{-1}x_j)}{(x_i-x_j)^2}\,.
\label{bottom}
\end{align}

Since for $v\in \F(u)$ we have
\begin{align*}
L_{\emptyset,h^\perp_{-r}v}
=[h^\perp_{-r},L_{\emptyset,v}]_{q^{r}}
\equiv (1-p^rq_2^{r/2})h^\perp_{-r}L_{\emptyset,v}\bmod \J_p\,,
\end{align*}
we obtain
\begin{align}
(\id\otimes \kappa) K\simeq\sum_{\lambda} \frac{\bra{\emptyset}h^\perp_\lambda}
{\bra{\emptyset}h^\perp_\lambda h^\perp_{-\lambda}\ket{\emptyset}}
\otimes 
\bigl(
\prod_{i=1}^{\ell(\lambda)}(1-p^{\lambda_i}
q_2^{\lambda_i/2})
h^\perp_{-\lambda_i}
\ast\one\bigr)
=\sum_{\lambda}
\frac{\bra{\emptyset}\alpha(
h^\perp_{\lambda}
)\otimes
\sigma(h^\perp_{-\lambda})\ast\one}
{\bra{\emptyset}h^\perp_{\lambda}
h^\perp_{-\lambda}\ket{\emptyset}},
\label{K3}
\end{align}
where $\alpha$ is an algebra homomorphism given by
\begin{align}
\alpha(h^\perp_r)=(1-p^rq_2^{r/2})h^\perp_r\,. 
\label{alpha}
\end{align}
Here $\simeq$ denotes equality  modulo vectors in $\F(u)^*\otimes \J_p$.

Recall the isomorphism $\sigma:\E_>\simeq Sh_0$ in Proposition \ref{prop:sigma}.
The following formula is known (\cite{FHSSY}, Proposition 1.12).
\begin{align}
\sum_{\lambda\in\cP}
\frac{h^\perp_{\lambda}\otimes \sigma(h^\perp_{-\lambda})}
{\bra{\emptyset}h^\perp_{\lambda}h^\perp_{-\lambda}\ket{\emptyset}}=
\sum_{\lambda\in\cP}
\frac{1}{(q_1-1)^{|\la|}}\frac{1}{\prod_{i=1}^{\ell(\lambda)}\lambda_i!}
\nu^*\bigl(m_\lambda\bigr)\otimes 
\epsilon^{(q_3)}_\lambda.
\label{Kn}
\end{align}

Combining \eqref{Kn} and \eqref{K3}, 
we arrive at the following. 

\begin{thm}\label{off-shell}
An off-shell Bethe vector in $\F(u)^*$ is given by
\begin{align*}
ev_a=
\sum_{\lambda\in\cP}
\frac{1}{(q_1-1)^{|\la|}}\frac{1}{\prod_{i=1}^{\ell(\lambda)}\lambda_i!}
ev_a\bigl(\epsilon^{(q_3)}_\lambda
\bigr)
\times 
\bra{\emptyset}\alpha\Bigl(\nu^*\bigl(m_\lambda\bigr)\Bigr)\,.
\end{align*}
Here $\epsilon^{(q_3)}$ is given by \eqref{bottom}, 
$\nu^*(m_\lambda)$ denotes the element of $\F^*(u)$ corresponding to
the monomial symmetric function $m_\la$, 
and $\alpha$ stands for the substitution \eqref{alpha}. 
\end{thm}

\section{Bethe ansatz in tensor product of Fock modules}\label{mult fock sec}
The method described above for diagonalizing $\tilde e_0^\perp(p)$ 
in the Fock module works for the case of other highest weight modules with straightforward modifications. Here we give some detail for the case of generic tensor products of Fock modules.

\medskip

Consider $V=\F(u_1)\otimes \F(u_2)\otimes\dots \otimes \F(u_k)$, where $u_1,\dots,u_k\in\C^\times$ are generic numbers. 
For generic $u_1,\dots,u_k$ this module is well defined and it is a bounded tame irreducible module, cf. \cite{FFJMM2}.
Set $\bs u=(u_1,\dots,u_k)$ and denote the lowest weight vector of $V$ by $\ket{\emptyset}$.

We define the corresponding shuffle algebra 
$Sh_{1}(\bs u)=\oplus_{n=0}^\infty Sh_{1,n}(\bs u)$. 
The space $Sh_{1,n}(\bs u)$ consists of all rational functions of the form
\begin{align*}
&F(x_1,\cdots,x_n)
=\frac{f(x_1,\cdots,x_n)}
{\prod_{1\le i<j\le n}(x_i-x_j)^2
\prod_{i=1}^n\prod_{j=1}^k(x_i-u_j)}\,,
\\
&f(x_1,\cdots,x_n)
\in \C[x_1^{\pm1},\cdots,x_n^{\pm1}]^{\GS_n}\,,
\end{align*}
such that they satisfy both the wheel condition \eqref{wheel1} 
and 
additional wheel conditions 
\begin{align*}
&f(u_i,q_2u_i,x_3,\cdots,x_n)=0\,, \qquad i=1,\dots, k.
\end{align*}
We denote the element $1\in Sh_{1,0}(\bs u)=\C$ by $\one$. 

Next we set
\begin{align*}
\phi(\bs u,x)=\prod_{i=1}^k\phi(u_i,x)=\prod_{i=1}^k\frac{qu_i-q^{-1}x}{u_i-x}\,
\end{align*}
and define the left and the right action of $Sh_{0}$ on $Sh_{1}(\bs u)$ by the same formulas as in the case $k=1$.

Similarly, we extend the left action of $Sh_{0}$ to the left action of $\E'$,
define $\J_0=Sh_{1}(\bs u)*Sh_{0}'$ and prove
\begin{prop}
We have an isomorphism  $Sh_{1}(\bs u)/\J_0\simeq V$
of left $\E'$ module sending $\one$ to $\ket{\emptyset}$.
\end{prop}
\begin{proof}
The crucial part is to do the Gordon filtration to find the size of the quotient.
It is done similarly to the case of $k=1$ discussed in Appendix \ref{gordon sec}.
The evaluation maps in the general case will depend on $k$ partitions, each starts evaluation in $u_i$. 
Since $u_i$ are generic, there is no interplay between different partitions.
\end{proof}

We define the map $\pi$  to be the projection map.

For the definition of $\kappa$ and the space $N$, 
it is not enough to use $h_{-r}^\perp$ only since 
$V$ is not a cyclic Heisenberg module. 
So, in addition to  \eqref{sh q com} we impose the condition, 
\begin{align*}
\kappa(e^\perp_{-n}v)=[e^\perp_{-n},\kappa(v)]_{q^{-n}}-
q^{-n}\sum_{j\ge1}\kappa(\psi^{+,\perp}_{j}v) e^\perp_{-n-j}, 
\end{align*}
for all $v\in V$, $n>0$, cf. \eqref{LL2}. Note that the sum on the right hand side is finite for any $v$.

Then $N$ is well defined and Corollary \ref{cyclic} still holds.

The vacuum to vacuum matrix element of $L$ operator now is just the product of operators in (\ref{Lee2}), 
$L_{\emptyset,\emptyset}(\bs u)=\prod_{i=1}^k L_{\emptyset,\emptyset}(u_i)$. 
The vacuum eigenvalue $\gamma_r$, cf. \eqref{h-eig}, for $V$ reads 
$\gamma_r=\frac{1-q_2^r}{r\kappa_r}\sum_{i=1}^k u_i^r$. 
With this change, Lemma \ref{LL-comm} remains valid.
We use the operator $L_{\emptyset,\emptyset}(\bs u)$ to define the map $\iota$ as in the case of $k=1$. 

We define the space of $p$ commutators $\J_p$ in the same way. And then Theorem \ref{main} holds just the same.
We introduce the Bethe equation 
\begin{align}
1=q^{-k}p\cdot \prod_{j=1}^k\frac{a_i-q_2 u_j}{a_i-u_j}
\prod_{j(\neq i)}\frac{(a_j-q_1a_i)(a_j-q_2a_i)(a_j-q_3a_i)}
{(a_j-q_1^{-1} a_i)(a_j-q_2^{-1} a_i)(a_j-q_3^{-1} a_i)}\,,
\quad i=1,\cdots,n\,.
\label{BAE k}
\end{align}
and arrive at the generalization of Theorem \ref{spectrum}.

\begin{thm}\label{spectrum k}
Let $a=(a_1,\cdots,a_n)$ be a solution of \eqref{BAE k} such that $ev_a$ is non-zero. 
Then the restriction of $ev_a$ to $N=\kappa(V)$ is an eigenvector in $V^*$ of the first integral 
of motion $\tilde{e}^\perp_0(p)$ \eqref{e0p} with the eigenvalue
\begin{align*}
E_1(a)=-\sum_{i=1}^n a_i+\frac{\sum_{i=1}^k u_i}{(1-q_1)(1-q_3)}\,.
\end{align*}
For generic $p$, $ev_a$ is a joint eigenvector of 
$\{I_{\F,n}(p)\}_{n=1}^\infty$. 
\end{thm}

\appendix

\section{Gordon filtration}\label{gordon sec}
In this section we study the size of the space $Sh_{1}(u)$ using 
the technique of the Gordon filtration.
Our goal is to prove Corollary \ref{Gordon-cor} below. 

Let $n$ be a positive integer, and let $\la$ be a partition such that $|\la|\le n$. 
For an element $F\in Sh_{1,n}(u)$, we introduce an operation of specialization 
$\rho_\la(F)$ 
as follows. 

First we set 
\begin{align*}
\rho^{(0)}_\la(F)(y_1,\cdots,y_{\ell(\la)},x_{|\la|+1},\cdots,x_n)
=
F(x_1,\ldots,x_n)
\Bigl|_{x_{\la_1+\cdots+\la_{i-1}+j}
=q_1^{j-1}y_i\quad (1\le i\le \ell(\la),\ 1\le j\le \la_i)}\,.
\end{align*}
The wheel condition \eqref{wheel1}
implies that $\rho^{(0)}_\la(F)$ is divisible by the factor 
\begin{align}
&\prod_{1\le a<b\le \ell(\la)} \prod_{\genfrac{}{}{0pt}{}{1\le i\le \la_a-1}{1\le j\le \la_b}}
(q_1^{j-1}y_b-q_1^iq_2y_a)(q_1^{j-1}y_b-q_1^iq_3y_a)
\prod_{k=|\la|+1}^n\prod_{\genfrac{}{}{0pt}{}{1\le i\le \ell(\la)}{1\le j\le \la_i-1}}(x_k-q_1^jq_2y_i)(x_k-q_1^jq_3y_i).
\label{rho-la1}
\end{align}
Next we set
\begin{align*}
\rho^{(1)}_\la(F)(y_2,\cdots,y_{\ell(\la)},x_{|\la|+1},\cdots,x_n)
=\Bigl[(y_1-u)
\rho^{(0)}_\la(F)(y_1,\cdots,y_{\ell(\la)},x_{|\la|+1},\cdots,x_n)
\Bigr]
\Bigl|_{y_1=u}.
\end{align*}
Then $\rho^{(1)}_\la(F)$ is divisible further by the factor
\begin{align*}
\prod_{k=|\la|+1}^n (x_k-q_2u)\,
\end{align*}
due to the wheel conditions \eqref{wheel1} and \eqref{wheel2}. 
For $i\ge 2$, we remove a factor contained in \eqref{rho-la1} to define
\begin{align*}
&\rho^{(i)}_\la(F)(y_{i+1},\cdots,y_{\ell(\la)},x_{|\la|+1},\cdots,x_n)\\
&=\Bigl[(y_i-q_3^{i-1}u)^{-\la_i+1}
\rho^{(i-1)}_\la(F)(y_{i},\cdots,y_{\ell(\la)},x_{|\la|+1},\cdots,x_n)
\Bigr]\Bigl|_{y_i=q_3^{i-1}u}.
\end{align*}
Finally we set $\rho_\la(F)=\rho^{(\ell(\la))}_\la(F)$. 

At each step, the wheel condition produces further factors. 
Collecting them together, we find that 
\begin{align}
\rho_\la(F)\in Sh_{0,n-|\la|}\times
\prod_{k=|\la|+1}^n\Bigl[
\prod_{(i,j)\in\lambda}\omega(x_k,q_3^{i-1}q_1^{j-1}u)\times
\frac{1}{\prod_{(i,j)\in CC(\la)}(x_k-q_3^{i-1}q_1^{j-1}u)}
\Bigr].
\label{rho-la2}
\end{align}

Let $\cP_{\lle n}$ denote the set of all partitions $\la$ with $|\la|\lle n$. 
Define a total ordering $>$ on $\cP_{\lle n}$ by setting 
$\mu>\la$ iff there is a $k$ such that $\mu_1=\la_1,\cdots,\mu_{k-1}=\la_{k-1}$, $\mu_k>\la_k$. 
We introduce a decreasing filtration $\{V_{n,\la}\}_{\la\in\cP_{\lle n}}$ on the space $V_n=Sh_{1,n}(u)$ by setting 
\begin{align*}
V_{n,\la}=\bigcap_{\mu>\la}\Ker\rho_\mu\subset V_n\,.
\end{align*}

\begin{prop}\label{rho prop} If $|\la|>m$ then
$\rho_\la(Sh_{1,m}(u)\ast Sh_{0,n-m})=0$.

If $|\la|=m$ then $\rho_\la(Sh_{1,m}(u)\ast Sh_{0,n-m})
=\rho_\la(V_{n,\la})$.

If $|\la|=n$ then $\rho_\la(
V_{n,\la}
)=\C$.
\end{prop}

\begin{proof}
The first statement is straightforward as all terms in the symmetrization of
the product $Sh_{1,m}(u)\ast
Sh_{0,n-m}$ clearly vanish under evaluation
$\rho_\la$ if $|\la|>m$.

Let $m=|\la|$. 
From the definition of the space $V_{n,\la}$ we see that, 
if $F$ is an element of $V_{n,\la}$, then 
$\rho_\la(F)$ contains an extra factor 
$\prod_{k=m+1}^n\prod_{(i,j)\in CC(\la)}(x_k-q_3^{i-1}q_1^{j-1}u)$
which cancels the denominator of \eqref{rho-la2}. 
Therefore we have 
\begin{align}
\rho_\la(V_{n,\la})\ \subset \ Sh_{0,n-m}\times
\prod_{k=m+1}^n
\Bigl[\prod_{(i,j)\in\lambda}\omega(x_k,q_3^{i-1}q_1^{j-1}u)\Bigr].
\label{rho-la3}
\end{align}

On the other hand, the following identity holds for
elements $G\in Sh_{1,m}(u)$ and $H\in Sh_{0,n-m}$: 
\begin{align}
\rho_\la(G\ast H)=const. \ 
\rho_\la(G)\cdot H(x_{m+1},\cdots,x_n)\times \prod_{\genfrac{}{}{0pt}{}{m+1\le k\le n}{(i,j)\in\la}}\omega(x_k,q_3^{i-1}q_1^{j-1}u)\,, 
\label{rho-la4}
\end{align}
where $const.$ is non-zero. 
We choose
\begin{align*}
&G=\e_{\la'}^{(q_1)}(x)  
\times \prod_{i=1}^m\frac{x_i-q_2 u}{x_i-u}\quad \in Sh_{1,m}(u),
\end{align*}
where $\la'=(\la'_1,\cdots,\la'_{l'})$ is the partition dual to $\la$ and 
$\e^{(q_1)}_{\la'}(x)$ is defined in \eqref{bottom}
where the parameter $q_3$ is changed to $q_1$.
It is easy to see that $G\in V_{m,\la}$, and that 
$\rho_\la(G)$ is a non-vanishing complex number.  
Since $H\in Sh_{0,n-m}$ is arbitrary in \eqref{rho-la4}, we conclude
that the inclusion in  
\eqref{rho-la3} is 
actually an equality and that
$\rho_\la(V_{n,\la})=\rho_\la\left(Sh_{1,m}(u)\ast Sh_{0,n-m}\right)$.

In particular, if $m=n$, then 
$\rho_\la(V_{n,\la})$ is a one dimensional vector space. 
\end{proof}

\begin{cor}\label{Gordon-cor}
The space $Sh_{1,n}(u)/\left(\sum_{m=0}^{n-1} Sh_{1,m}(u)\ast Sh_{0,n-m}
\right)$ 
is a finite dimensional vector space of dimension $p(n)$, 
the number of partitions of $n$. 
\end{cor}
\begin{proof}
We consider the associated graded space related to the
filtration $\{V_{n,\la}\}$ of the space $V_n=Sh_{1,n}(u)$. By the definition, we have $\on{gr}_\la(V_n)=\rho_\la(V_{n,\la})$. Now, for the factor space,  from Proposition \ref{rho prop} we have $\on{gr}_\la(Sh_{1,n}(u)/\left(\sum_{m=0}^{n-1} Sh_{1,m}(u)\ast Sh_{0,n-m}
\right)$ is zero if $|\la|<n$ and one dimensional if $|\la|=n$. The corollary follows.
\end{proof}

\bigskip

 {\bf Acknowledgments.}
We would like to thank Jun'ichi Shiraishi for enlightening discussions. 

The financial support from the Government of the Russian Federation within the framework of the implementation of the 5-100 Programme Roadmap of the National Research University  Higher School of Economics is acknowledged.
Research of MJ is supported by the Grant-in-Aid 
for Scientific Research B-23340039.

EM and BF would like to thank Kyoto University and Rikkyo University for hospitality during their visits when the main part of this work was completed.


\begin{thebibliography}{0000000}
\bibitem[AL]{AL}  M. Alfimov, A. Litvinov, 
{\it On spectrum of ILW hierarchy in conformal field theory II: coset CFT's},
arXiv:1411.3313, 1-11 

\bibitem[BS]{BS}
I. Burban and O. Schiffmann, 
{\it On the Hall algebra of an elliptic curve I},
Duke Math. J. {\bf 161} (2012), no.7, 
1171--1231

\bibitem[FFJMM1]{FFJMM1}
B. Feigin, E. Feigin, 
M. Jimbo, T. Miwa and E. Mukhin,
{\it Quantum continuous $\mathfrak{gl}_\infty$:
Semi-infinite construction of representations},
Kyoto J. Math. {\bf 51} (2011), no. 2, 
337--364

\bibitem[FFJMM2]{FFJMM2}
B. Feigin, E. Feigin, M. Jimbo, T. Miwa and E. Mukhin,
{\it Quantum continuous $\mathfrak{gl}_\infty$:
Tensor product of Fock modules and $\mathcal{W}_n$ characters},
Kyoto J. Math. {\bf 51} (2011), no. 2, 
365--392 

\bibitem[FHHSY]{FHHSY}
B. Feigin, K. Hashizume, A. Hoshino, J. Shiraishi 
and S. Yanagida,
{\it A commutative algebra on degenerate $\mathbb{C}P^1$ and
Macdonald polynomials}, J. Math. Phys. {\bf 50} (2009), no. 9, 095215, 1--42
 
\bibitem[FHSSY]{FHSSY}
B. Feigin, A. Hoshino, J. Shibahara, J. Shiraishi 
and S. Yanagida,
{\it Kernel function and quantum algebras}, 
arXiv:1002.2485

\bibitem[FJMM1]{FJMM1} B. Feigin, M. Jimbo, T. Miwa, and E. Mukhin, 
{\it Quantum toroidal $\mathfrak{gl}_1$ algebra : plane partitions}, 
{\it Kyoto J. Math.} \textbf{52} (2012), no.3, 621--659

\bibitem[FKSW]{FKSW}
B. Feigin, T. Kojima, J. Shiraishi and H. Watanabe, 
{\it The integrals of motion for the deformed Virasoro algebra}
arXiv:0705.0427v2

\bibitem[FKSW2]{FKSW2}
B. Feigin, T. Kojima, J. Shiraishi and H. Watanabe, 
{\it The integrals of motion for the deformed $W$ algebra $W_{q,t}(\widehat{sl}_N)$},
arXiv:0705.0627v1

\bibitem[FO]{FO}
B.~Feigin and 
A.~Odesskii, {\it Vector bundles on an elliptic curve and Sklyanin algebras},
Topics in quantum groups and finite-type invariants, 
Amer. Math. Soc. Transl. Ser. 2, {\bf 185} (1998), 
65--84

\bibitem[FT]{FT}
B.~Feigin and A.~Tsymbaliuk,
{\it Heisenberg action in the equivariant
K-theory of Hilbert schemes via Shuffle Algebra},
Kyoto J. Math.  {\bf 51}  (2011),  no. 4, 831--854


\bibitem[KS]{KS}
T. Kojima and J. Shiraishi, 
{\it The integrals of motion for the deformed $W$ algebra $W_{q,t}(\widehat{gl}_N)$. II. 
Proof of the commutation relations}
Commun. Math. Phys. {\bf 283} (2008), no. 3,
795--851

\bibitem[L]{L} A. V. Litvinov, 
{\it On spectrum of ILW hierarchy in conformal field theory}, 
JHEP, {\bf 11} (2013), 155 

\bibitem[M]{M} K. Miki, {\it A $(q,\gamma)$ analog of the $W_{1+\infty}$ algebra}, 
Journal of Math. Phys. {\bf 48} (2007), no. 12, 1--35

\bibitem[NS]{NS} N. Nekrasov and S. Shatashvili, 
{\it Supersymmetric vacua and Bethe ansatz}, 
Nucl. Phys. Proc. Suppl. {\bf 192-193} (2009), 91--112


\bibitem[Ng]{Ng} A. Negut, 
{\it The shuffle algebra revisited}, 
Int. Math. Res. Notices {\bf 22} (2014), 
6242--6275


\bibitem[S]{S} O.~Schiffmann,
{\it Drinfeld realization of the elliptic Hall algebra}, 
 J. Algebraic Combin. {\bf 35}  (2012),  no. 2, 237--262


\bibitem[Sa1]{Sa1} 
Yosuke Saito, 
{\it Elliptic Ding-Iohara algebra and 
the free field realization of the elliptic Macdonald operator},
Publ. Res. Inst. Math. Sci.  {\bf 50}  (2014),  no. 3, 411–-455

\bibitem[Sa2]{Sa2} 
Yosuke Saito, 
{\it Commutative families of the elliptic Macdonald operator}, 
SIGMA {\bf 10} (2014), 021

\bibitem[SV1]{SV1}
O. Schiffmann and E. Vasserot,
{\it The elliptic Hall algebra, Cherednik Hecke algebras and Macdonald
polynomials}, Compos. Math. {\bf 147} (2011), no. 1, 188--234

\bibitem[SV2]{SV2}
O.~Schiffmann and E.~Vasserot,
{\it The elliptic Hall algebra and the equivariant K-theory of the Hilbert scheme of $\mathbb{A}^2$},
Compos. Math. {\bf 147}  (2011),  no. 1, 188--234

\end{thebibliography}
\end{document}